\documentclass[fleqn,a4paper,9pt]{article}
\usepackage{amsmath}
\usepackage{mathtools}
\usepackage{mathptmx}
\usepackage{comment}
\usepackage{tabularx}
\usepackage{listings}
 \usepackage{amsthm} 
\usepackage{amssymb}

\makeatletter
\let\@@pmod\pmod
\DeclareRobustCommand{\pmod}{\@ifstar\@pmods\@@pmod}
\def\@pmods#1{\mkern4mu({\operator@font mod}\mkern 6mu#1)}
\makeatother
\DeclarePairedDelimiter\floor{\lfloor}{\rfloor}
\newcommand{\tpmod}[1]{\mkern 8mu({\operator@font mod}\mkern 6mu#1)}
\usepackage{atbegshi}
\AtBeginDocument{\AtBeginShipoutNext{\AtBeginShipoutDiscard}}

\author{Leon Fairbanks}
\begin{document}
\begin{flushleft}
\title{{\textbf{Powers of Cosine and Sine}}}
\maketitle
\begin{abstract}
We consider objects of the form
$$\cos^m\left(\frac{(2i-1)\pi}{2^n}\right)$$\notag
\end{abstract}
\tableofcontents

\newtheorem{theorem}{Theorem}[section]
\newtheorem{lemma}[theorem]{Lemma}
\newtheorem{proposition}[theorem]{Proposition}
\newtheorem{corollary}[theorem]{Corollary}
\newtheorem{definition}[theorem]{Definition}

\section{Introduction}
\begin{definition}
Let $V_n$ be the set of terms of the form $\sin\left(\frac{i\pi}{2^n})\right)$ where $i\in\mathbb{Z}$,  $1\le i\le 2^{n-1}$. Call the members of $V_n$ the elementary terms. 
\end{definition}
Then we let $E_n$ be the extension of $Q$ generated by adjoining elements of $V_n$. Trig functions show that any product of terms in $E_n$ can be written as a sum of elementary terms with rational coefficients. Hence the degree of the extension $E_n$ is $\le 2^{n-1}$. Note, since $\sin(\frac{\pi}{4})=\cos(\frac{\pi}{4})=\frac{\sqrt{2}}{2}$ and $\sin(\frac{\theta}{2})=\pm\sqrt{\frac{1-\cos(\theta)}{2}}$ these terms have a nice expression involving nested roots of 2 as in,
$$\sin(\frac{7\pi}{16})=\frac{\sqrt{2+\sqrt{2+\sqrt{2}}}}{2}$$
\begin{lemma}
Assume $r,k\in \mathbb{Z}$, $1\le k\le r$, then
\begin{align}\notag
\sum_{i=0}^{r}(-1)^{i}\frac{2^{2i-1}}{r+i}\binom{r+i}{r-i}\binom{2i}{k}&=\frac{2^{k}}{2r+k}\binom{2r+k}{2r-k}\\\notag
\end{align}
\end{lemma}
\begin{proof}
Using Zeilberger's algorithm we find the sum satisfies the following recursion:
\begin{align}\notag
\rm{SUM}[1+r]=-\frac{(k+2r)(12+k+2r)}{(-2+k-2r)(-1+k-2r)}\rm{SUM}[r]\\\notag
\end{align}
We then verify that the formula on the right satisfies this recursion and has the proper initial values.
\end{proof}
\begin{corollary}
\begin{align}\notag
\sum_{i=0}^{2^{n-2}}(-1)^{i}\frac{2^{2i-1}}{2^{n-2}+i}\binom{2^{n-2}+i}{2^{n-2}-i}\binom{2i}{k}&=\frac{2^{k}}{2^{n-1}+k}\binom{2^{n-1}+k}{2^{n-1}-k}\\\notag
\end{align}
\end{corollary}

\begin{proposition}
The minimal polynomial for $\cos(\frac{(2i-1)\pi}{2^n})$ is
$$f_{n}(x)=((\cdots((((2x)^2-2)^2-2)^2-2)^2-2\cdots)^2-2)/2$$ 
where the degree of $f_{n}(x)$ is $2^{n-1}$. 
In $f(x)$, we let $c_{n,m}$ be the coefficient of $x^{2m}$ (the $m+1$-st term), then
$$f_{n}(x)=1+c_{n,1}x^2+c_{n,2}x^4+\cdots+c_{n,2^{n-2}}x^{2^{n-1}}$$
where
\begin{align}\notag
c_{n,m}&=(-1)^m \frac{ 2^{n+2m-2}} {2^{n-2}+m}\binom{2^{n-2}+m}{2^{n-2}-m},\ \ \  1\le m\le2^{n-2}\\\notag
\end{align}
\end{proposition}
\begin{proof}
The form of the first expression for $f$ implies that $f_n(2x^2-1)=f_{n+1}(x)$.
\begin{align}\notag
f_n(2x^2-1)&=\sum_{m=0}^{2^{n-2}}(-1)^m \frac{ 2^{n+2m-2}} {2^{n-2}+m}\binom{2^{n-2}+m}{2^{n-2}-m}(2x^2-1)^{2m}\\\notag
&=\sum_{m=0}^{2^{n-2}}(-1)^m \frac{ 2^{n+2m-2}} {2^{n-2}+m}\binom{2^{n-2}+m}{2^{n-2}-m}\sum_{k=0}^{2m}(-1)^k(2x^2)^{k}\\\notag
&=\sum_{k=0}^{2^{n-2}}\sum_{m=0}^{2^{n-2}}(-1)^m\frac{2^{n+2m-2}}{2^{n-2}+m}\binom{2^{n-2}+m}{2^{n-2}-m}(-1)^k\binom{2m}{k}(2x^2)^k\\\notag
&=\sum_{k=0}^{2^{n-2}}(-1)^k2^{k+n-1}\sum_{m=0}^{2^{n-2}}(-1)^m\frac{2^{2m-1}}{2^{n-2}+m}\binom{2^{n-2}+m}{2^{n-2}-m}\binom{2m}{k}x^{2k}\\\notag
&=\sum_{m=0}^{2^{n-1}}(-1)^m \frac{ 2^{n+2m-1}} {2^{n-1}+m}\binom{2^{n-1}+m}{2^{n-1}-m}x^{2m}\\\notag
&=f_{n+1}(x)\\\notag
\end{align}
By the previous lemma.

\end{proof}
The roots of the minimal polynomial are elementary terms of the form $\cos\left(\frac{(2i-1)\pi}{2^n})\right)$, each appearing with a "+" and a "-" sign. The extension is normal, the $2^{n-1}$ powers of a root  form a basis of the extension which has rank $2^{n-1}$.  \\
Note if you replace each $x^k$ with $x^{r*k}$, where $r\in \mathbb{Z}^+$, this becomes the minimal polynomial for  $(\cos(\frac{i\pi}{2^n}))^{1/r}$.\\
Example: $n=5$
$$\cos(\frac{\pi}{32})=\frac{\sqrt{2+\sqrt{2+\sqrt{2+\sqrt{2}}}}}{2}$$
\begin{align}\notag
f_n(x)=&(((((2x)^2-2)^2-2)^2-2)^2-2)/2\\\notag
=&1-128 x^2+2688 x^4-21504 x^6+84480 x^8-180224 x^{10}+212992 x^{12}-131072 x^{14}+32768 x^{16}\\\notag
=&2^{15}(x-\frac{\sqrt{2-\sqrt{2+\sqrt{2+\sqrt{2}}}}}{2})(x+\frac{\sqrt{2-\sqrt{2+\sqrt{2+\sqrt{2}}}}}{2})(x-\frac{\sqrt{2-\sqrt{2+\sqrt{2-\sqrt{2}}}}}{2})\\\notag
&(x+\frac{\sqrt{2-\sqrt{2+\sqrt{2-\sqrt{2}}}}}{2})(x-\frac{\sqrt{2-\sqrt{2-\sqrt{2-\sqrt{2}}}}}{2})(x+\frac{\sqrt{2-\sqrt{2-\sqrt{2-\sqrt{2}}}}}{2})\\\notag
&(x-\frac{\sqrt{2-\sqrt{2-\sqrt{2+\sqrt{2}}}}}{2})(x+\frac{\sqrt{2-\sqrt{2-\sqrt{2+\sqrt{2}}}}}{2})(x-\frac{\sqrt{2+\sqrt{2-\sqrt{2+\sqrt{2}}}}}{2})\\\notag
&(x+\frac{\sqrt{2+\sqrt{2-\sqrt{2+\sqrt{2}}}}}{2})(x-\frac{\sqrt{2+\sqrt{2-\sqrt{2-\sqrt{2}}}}}{2})(x+\frac{\sqrt{2+\sqrt{2-\sqrt{2-\sqrt{2}}}}}{2})\\\notag
&(x-\frac{\sqrt{2+\sqrt{2+\sqrt{2-\sqrt{2}}}}}{2})(x+\frac{\sqrt{2+\sqrt{2+\sqrt{2-\sqrt{2}}}}}{2})(x-\frac{\sqrt{2+\sqrt{2+\sqrt{2+\sqrt{2}}}}}{2})\\\notag
&(x+\frac{\sqrt{2+\sqrt{2+\sqrt{2+\sqrt{2}}}}}{2})\notag
\end{align}

The set of elementary terms, $$\{\sin(\frac{i\pi}{2^n})\ \ |\ 0<i< 2^{n-1}\}$$ equals the set $$\{\cos(\frac{i\pi}{2^n})\ |\ 0<i< 2^{n-1}\}$$ 
The odd powers of $\cos\left(\frac{(2i-1)\pi}{2^n}\right)$ (or of $\sin\left(\frac{(2i-1)\pi}{2^n}\right)$) can be expressed in terms of the subset of $V_n$ of terms of the form $\cos\left(\frac{(2i-1)\pi}{2^n}\right)$, the roots of the minimal polynomial. This paper presents expressions for odd and even powers of cosine in terms of the elementary terms.

\section{Odd Powers Of Cosine}

In general, for $n\ge 0,\ n,r\in\mathbb{Z}$, $r$ odd, there is a $2^{n-2}x2^{n-2}$ matrix $M$ such that
\begin{equation}\notag
\left(
\begin{array}{cc}
\cos^r\left(\frac{\pi}{2^n}\right)   \\
\cos^r\left(\frac{3\pi}{2^n}\right)  \\
\vdots\\
\cos^r\left(\frac{(2^{n-2}-1)\pi}{2^n}\right)  \\
\end{array}
\right)=M\left(
\begin{array}{cc}
\cos\left(\frac{\pi}{2^n}\right)   \\
\cos\left(\frac{3\pi}{2^n}\right)  \\
\vdots\\
\cos\left(\frac{(2^{n-2}-1)\pi}{2^n}\right)  \\
\end{array}
\right)
\end{equation}

For example, for $n=4,\ r=15$
\begin{equation}\notag
\left(
\begin{array}{cc}
\cos^{15}\left(\frac{\pi}{16}\right)   \\
\cos^{15}\left(\frac{3\pi}{16}\right)  \\
\cos^{15}\left(\frac{5\pi}{16}\right)  \\
\cos^{15}\left(\frac{7\pi}{16}\right)  \\
\end{array}
\right)=\frac{1}{2^{14}}\left(
\begin{array}{cccc}
 6434 & 4990 & 2898 & 910 \\
 -2898 & 6434 & -910 & -4990 \\
 -4990 & 910 & 6434 & 2898 \\
 -910 & 2898 & -4990 & 6434 \\
\end{array}
\right)\left(
\begin{array}{cc}
\cos\left(\frac{\pi}{16}\right)   \\
\cos\left(\frac{3\pi}{16}\right)  \\
\cos\left(\frac{5\pi}{16}\right)  \\
\cos\left(\frac{7\pi}{16}\right)  \\
\end{array}
\right)
\end{equation}
While  for $n=4,\ r=-3$
\begin{equation}\notag
\left(
\begin{array}{cc}
\cos^{-3}\left(\frac{\pi}{16}\right)   \\
\cos^{-3}\left(\frac{3\pi}{16}\right)  \\
\cos^{-3}\left(\frac{5\pi}{16}\right)  \\
\cos^{-3}\left(\frac{7\pi}{16}\right)  \\
\end{array}
\right)=2^3
\left(
\begin{array}{cccc}
 2 & -5 & 7 & -8 \\
 -7 & 2 & 8 & 5 \\
 5 & -8 & 2 & 7 \\
 8 & 7 & 5 & 2 \\
\end{array}
\right)
\left(
\begin{array}{cc}
\cos\left(\frac{\pi}{16}\right)   \\
\cos\left(\frac{3\pi}{16}\right)  \\
\cos\left(\frac{5\pi}{16}\right)  \\
\cos\left(\frac{7\pi}{16}\right)  \\
\end{array}
\right)
\end{equation}
Note that given a matrix equation relating powers of cosine to a cosine basis  as above, we can get the analogous equations involving sine by reversing the order of the rows and reversing the order of the columns.

In the following, we assume  $i, j, n,r\in\mathbb{Z^+}$.
\begin{lemma}
The set of elements $\cos^j(\frac{(2i-1)\pi}{2^n})$, $0\le j\le 2^{n-1}-1$ constitute a basis for the extension obtained from adjoining $\cos(\frac{(2i-1)\pi}{2^n})$ to $\mathbb{Q}$. An endomorphism of the extension is obtained by sending $\cos(\frac{(2i-1)\pi}{2^n})$ to $\cos\left(\frac{(2j-1)\pi}{2^n}\right)$.
\end{lemma}
\begin{proof}
(Theory of field extensions)
\end{proof}
\begin{lemma}
\begin{align}\notag
\sum_{j=1}^i (-1)^j 2^{2j-1}\frac{2i-1}{2j-1}\binom{i+j-2}{2j-2}&=(-1)^i2\\\notag
\end{align}
\end{lemma}
\begin{proof}
Let $$f(i)=\sum_{j=1}^i (-1)^j 2^{2j-1}\frac{2i-1}{2j-1}\binom{i+j-2}{2j-2}$$
Then Zeilberger's Algorithm shows $$f(i)+f(i+1)=0$$
while $f(1)=-2$.
\end{proof}
\begin{lemma}
Assume $i\ge j\ge 2$, then
\begin{align}\notag
\sum_{k=1}^i (-1)^k 2^{2k}\frac{2i-1}{2k-1}\binom{i+k-2}{2k-2}\binom{k}{j-1}&=(-1)^i 2^{2j-2}\frac{2i^2-2i+j-1}{(2j-3)(j-1)}\binom{i+j-3}{i-j+1}\\\notag
\end{align}
\end{lemma}
\begin{proof}
Note, for $i=j=2$ the formula is true. We apply Zeiberger's algorithm to the sum $$S(i)=\sum_{k=1}^i (-1)^k 2^{2k}\frac{2i-1}{2k-1}\binom{i+k-2}{2k-2}\binom{k}{j-1}$$ The response is: If $i-2$ is a natural number, then:
\begin{align}\notag
-(i+j-2)(2i^2+2i+j-1)S(i)-(i-j+2)(2i^2-2i+j-1)S(i+1)&=-8(j-2)(j-1)\binom{1}{j-1}\\\notag
\end{align}
Let
\begin{align}\notag
f(i)&=(-1)^i 2^{2j-2}\frac{2i^2-2i+j-1}{(2j-3)(j-1)}\binom{i+j-3}{i-j+1}\\\notag
\end{align}
Then
\begin{align}\notag
&-(i+j-2)(2i^2+2i+j-1)f(i)-(i-j+2)(2i^2-2i+j-1)f(i+1)+8(j-2)(j-1)\binom{1}{j-1}\\\notag
&=\frac{8}{\pi}\sin(j\pi)\\\notag
\end{align}
Hence the formula satisfies the recursion for the sum and has the correct initial value.
\end{proof}
In the following proposition, the polynomials $p_i(x)$ are the Chebyshev polynomials of the first kind of odd order.
\begin{proposition}\label{chebyshev}
Assume $i\ge 1, n\ge 2$. Let $$p_i(x)=\sum_{j=1}^{i}(-1)^{j}2^{2j-2}\binom{i+j-2}{2j-2}\frac{2i-1}{2j-1}x^{2j-1}$$
then 
\begin{align}\notag
\sin\left(\frac{(2i-1)\pi}{2^n}\right) &=-p_i(\sin\left(\frac{\pi}{2^n}\right) )\\\notag
\cos\left(\frac{(2i-1)\pi}{2^n}\right) &=(-1)^i p_i(\cos\left(\frac{\pi}{2^n}\right) )\\\notag
\end{align}
\end{proposition}
\begin{proof}
We begin with the cosine formula. We start with induction on $i$. The statement is true for $i=1$,$j=1$. We assume the following is true for fixed value $i$:
\begin{align}\notag
\sin\left(\frac{(2i-1)\pi}{2^n}\right) &=-\sum_{j=1}^i (-1)^j 2^{2j-2}\binom{i+j-2}{2j-2}\frac{2i-1}{2j-1}\sin^{2j-1}\left(\frac{\pi}{2^n}\right)\\\notag
\cos\left(\frac{(2i-1)\pi}{2^n}\right) &=(-1)^i\sum_{j=1}^i (-1)^j 2^{2j-2}\binom{i+j-2}{2j-2}\frac{2i-1}{2j-1}\cos^{2j-1}\left(\frac{\pi}{2^n}\right)\\\notag
\end{align}
We want to show 
\begin{align}\notag
\cos\left(\frac{(2i+1)\pi}{2^n}\right) &=(-1)^{i+1}\sum_{j=1}^{i+1} (-1)^j 2^{2j-2}\binom{i+j-1}{2j-2}\frac{2i+1}{2j-1}\cos^{2j-1}\left(\frac{\pi}{2^n}\right)\\\notag
\end{align}
\begin{align}\notag
&\cos\left(\frac{(2i-1)\pi}{2^n}+\frac{2\pi}{2^n}\right)=\cos(\frac{(2i-1)\pi}{2^n})\cos(\frac{2\pi}{2^n})-\sin(\frac{(2i-1)\pi}{2^n})\sin(\frac{2\pi}{2^n})\\\notag
&=(-1)^i\sum_{j=1}^i (-1)^j 2^{2j-2}\binom{i+j-2}{2j-2}\frac{2i-1}{2j-1}\cos^{2j-1}\left(\frac{\pi}{2^n}\right)\cos(\frac{2\pi}{2^n})\\\notag
&+\sum_{j=1}^i (-1)^j 2^{2j-2}\binom{i+j-2}{2j-2}\frac{2i-1}{2j-1}\sin^{2j-1}\left(\frac{\pi}{2^n}\right)\sin(\frac{2\pi}{2^n})\\\notag
&=\sum_{j=1}^i (-1)^j 2^{2j-2}\binom{i+j-2}{2j-2}\frac{2i-1}{2j-1}\biggl((-1)^i\cos^{2j-1}\left(\frac{\pi}{2^n}\right)\cos(\frac{2\pi}{2^n})\\\notag
&+\sin^{2j-1}\left(\frac{\pi}{2^n}\right)\sin(\frac{2\pi}{2^n})\biggr)\\\notag
&=\sum_{j=1}^i (-1)^j 2^{2j-2}\binom{i+j-2}{2j-2}\frac{2i-1}{2j-1}\biggl((-1)^i\cos^{2j-1}\left(\frac{\pi}{2^n}\right)(2\cos^2(\frac{\pi}{2^n})-1)\\\notag
&+\sin^{2j}\left(\frac{\pi}{2^n}\right)(2\cos(\frac{2\pi}{2^n}))\biggr)\\\notag
&=\sum_{j=1}^i (-1)^j 2^{2j-2}\binom{i+j-2}{2j-2}\frac{2i-1}{2j-1}\biggl((-1)^i\cos^{2j-1}\left(\frac{\pi}{2^n}\right)(2\cos(\frac{\pi}{2^n})-1)\\\notag
&+(1-\cos^{2}\left(\frac{\pi}{2^n}\right))^j(2\cos(\frac{2\pi}{2^n}))\biggr)\\\notag
&=\sum_{j=1}^i (-1)^j 2^{2j-2}\binom{i+j-2}{2j-2}\frac{2i-1}{2j-1}\biggl((-1)^i(2\cos^{2j+1}\left(\frac{\pi}{2^n}\right)-\cos^{2j-1}\left(\frac{\pi}{2^n}\right))\\\notag
&+2\sum_{k=0}^j(-1)^k\binom{j}{k}\cos^{2k}(\frac{2\pi}{2^n})\cos(\frac{2\pi}{2^n})\biggr)\\\notag
&=(-1)^i (2i-1)\cos(\frac{2\pi}{2^n})+2^{2i-1}\cos^{2i+1}(\frac{2\pi}{2^n})\\\notag
&+\sum_{j=1}^{i-1}(-1)^{i+j}2^{2j-1}(2i-1)\left(\frac{1}{2j-1}\binom{i+j-2}{2j-2}+\frac{2}{2j+1}\binom{i+j-1}{2j}\right)\cos^{2j+1}(\frac{2\pi}{2^n})\\\notag
&+\sum_{j=1}^i(-1)^j2^{2j-1}\frac{2i-1}{2j-1}\binom{i+j-2}{2j-2}\sum_{k=0}^j\binom{j}{k}\cos^{2k+1}(\frac{2\pi}{2^n})\\\notag
&=((-1)^i (2i-1)+\sum_{j=1}^i(-1)^j2^{2j-1}\frac{2i-1}{2j-1}\binom{i+j-2}{2j-2} )\cos(\frac{\pi}{2^n})+2^{2i-1}\cos^{2i+1}(\frac{\pi}{2^n})\\\notag
&+\sum_{j=1}^{i-1}(-1)^{i+j}2^{2j-1}(2i-1)\left(\frac{1}{2j-1}\binom{i+j-2}{2j-2}+\frac{2}{2j+1}\binom{i+j-1}{2j}\right)\cos^{2j+1}(\frac{\pi}{2^n})\\\notag
&+\sum_{j=1}^i(-1)^j2^{2j-1}\frac{2i-1}{2j-1}\binom{i+j-2}{2j-2}\sum_{k=1}^j\binom{j}{k}\cos^{2k+1}(\frac{\pi}{2^n})\\\notag
\end{align}
\begin{align}\notag
&=\biggl((-1)^i (2i-1)+\sum_{j=1}^i(-1)^j2^{2j-1}\frac{2i-1}{2j-1}\binom{i+j-2}{2j-2}\biggr )\cos(\frac{\pi}{2^n})+2^{2i}\cos^{2i+1}(\frac{\pi}{2^n})\\\notag
&+\sum_{j=1}^{i-1}(-1)^{i+j}2^{2j-1}(2i-1)\left(\frac{1}{2j-1}\binom{i+j-2}{2j-2}+\frac{2}{2j+1}\binom{i+j-1}{2j}\right)\cos^{2j+1}(\frac{\pi}{2^n})\\\notag
&+\sum_{j=1}^{i-1}(-1)^j2^{2j-1}\frac{2i-1}{2j-1}\binom{i+j-2}{2j-2}\sum_{k=1}^j\binom{j}{k}\cos^{2k+1}(\frac{\pi}{2^n})\\\notag
&+(-1)^i2^{2i-1}\sum_{k=1}^{i-1}(-1)^k\binom{i}{k}\cos^{2k+1}(\frac{\pi}{2^n})\\\notag
&=\biggl((-1)^i (2i-1)+\sum_{j=1}^i(-1)^j2^{2j-1}\frac{2i-1}{2j-1}\binom{i+j-2}{2j-2}\biggr )\cos(\frac{\pi}{2^n})+2^{2i}\cos^{2i+1}(\frac{\pi}{2^n})\\\notag
&+\sum_{j=1}^{i-1}\biggl((-1)^{i+j}2^{2j-1}(2i-1)\left(\frac{1}{2j-1}\binom{i+j-2}{2j-2}+\frac{2}{2j+1}\binom{i+j-1}{2j}\right)+\\\notag
&+\sum_{k=1}^{i-1}(-1)^k2^{2k-1}\frac{2i-1}{2k-1}\binom{i+k-2}{2k-2}(-1)^j\binom{k}{j}\biggr)\cos^{2j+1}(\frac{\pi}{2^n})\\\notag
&+(-1)^i2^{2i-1}\sum_{k=1}^{i-1}(-1)^k\binom{i}{k}\cos^{2k+1}(\frac{\pi}{2^n})\\\notag
&=\biggl((-1)^i (2i-1)+\sum_{j=1}^i(-1)^j2^{2j-1}\frac{2i-1}{2j-1}\binom{i+j-2}{2j-2}\biggr )\cos(\frac{\pi}{2^n})+2^{2i}\cos^{2i+1}(\frac{\pi}{2^n})\\\notag
&+\sum_{j=1}^{i-1}\biggl((-1)^{i+j}2^{2j-1}(2i-1)\left(\frac{1}{2j-1}\binom{i+j-2}{2j-2}+\frac{2}{2j+1}\binom{i+j-1}{2j}\right)+\\\notag
&+(-1)^{i+j}2^{2i-1}\binom{i}{j}+\sum_{k=1}^{i-1}(-1)^{j+k}2^{2k-1}\frac{2i-1}{2k-1}\binom{i+k-2}{2k-2}\binom{k}{j}\biggr)\cos^{2j+1}(\frac{\pi}{2^n})\\\notag
&=\biggl((-1)^i (2i-1)+\sum_{j=1}^i(-1)^j2^{2j-1}\frac{2i-1}{2j-1}\binom{i+j-2}{2j-2}\biggr )\cos(\frac{\pi}{2^n})+2^{2i}\cos^{2i+1}(\frac{\pi}{2^n})\\\notag
&+\sum_{j=2}^{i}\biggl((-1)^{i+j-1}2^{2j-3}(2i-1)\left(\frac{1}{2j-3}\binom{i+j-3}{2j-4}+\frac{2}{2j-1}\binom{i+j-2}{2j-2}\right)+\\\notag
&+(-1)^{i+j-1}2^{2i-1}\binom{i}{j-1}+\sum_{k=1}^{i-1}(-1)^{j+k-1}2^{2k-1}\frac{2i-1}{2k-1}\binom{i+k-2}{2k-2}\binom{k}{j-1}\biggr)\cos^{2j-1}(\frac{\pi}{2^n})\\\notag
&=\biggl((-1)^i (2i-1)+\sum_{j=1}^i(-1)^j2^{2j-1}\frac{2i-1}{2j-1}\binom{i+j-2}{2j-2}\biggr )\cos(\frac{\pi}{2^n})+2^{2i}\cos^{2i+1}(\frac{\pi}{2^n})\\\notag
&+\sum_{j=2}^{i}(-1)^{i+j-1}2^{2j-3}\biggl((2i-1)\left(\frac{1}{2j-3}\binom{i+j-3}{2j-4}+\frac{2}{2j-1}\binom{i+j-2}{2j-2}\right)+\\\notag
&+2^{2i-2j+2}\binom{i}{j-1}+ \frac{2i^2-2i+j-1}{(2j-3)(j-1)}\binom{i+j-3}{i-j+1}-2^{2i-2j+2}\binom{i}{j-1}\biggr)\cos^{2j-1}(\frac{\pi}{2^n})\\\notag
\end{align}
\begin{align}\notag
&=\biggl((-1)^i (2i-1)+\sum_{j=1}^i(-1)^j2^{2j-1}\frac{2i-1}{2j-1}\binom{i+j-2}{2j-2}\biggr )\cos(\frac{\pi}{2^n})+2^{2i}\cos^{2i+1}(\frac{\pi}{2^n})\\\notag
&+\sum_{j=2}^{i}(-1)^{i+j-1}2^{2j-3}\biggl(\left(\frac{2i-1}{2j-3}+\frac{2i^2-2i+j-1}{(2j-3)(j-1)}\right)\binom{i+j-3}{2j-4}\\\notag
&+\frac{2(2i-1)}{2j-1}\binom{i+j-2}{2j-2}\biggr)\cos^{2j-1}(\frac{\pi}{2^n})\\\notag
&=\biggl((-1)^i (2i-1)+\sum_{j=1}^i(-1)^j2^{2j-1}\frac{2i-1}{2j-1}\binom{i+j-2}{2j-2}\biggr )\cos(\frac{\pi}{2^n})+2^{2i}\cos^{2i+1}(\frac{\pi}{2^n})\\\notag
&+\sum_{j=2}^{i}(-1)^{i+j-1}2^{2j-2}\biggl(\left(\frac{i^2-2i+ij}{(2j-3)(j-1)}+\frac{(i+j-2)(i-j+1)}{(2j-2)(2j-3)}\right)\binom{i+j-3}{2j-4}\biggr)\\\notag
&\cos^{2j-1}(\frac{\pi}{2^n})\\\notag
&=\biggl((-1)^i (2i-1)+\sum_{j=1}^i(-1)^j2^{2j-1}\frac{2i-1}{2j-1}\binom{i+j-2}{2j-2}\biggr )\cos(\frac{\pi}{2^n})+2^{2i}\cos^{2i+1}(\frac{\pi}{2^n})\\\notag
&+\sum_{j=2}^{i}(-1)^{i+j-1}2^{2j-2}\frac{2+i-5i^2+2i^3-3j-4ij+4i^2j+j^2+2ij^2}{(2j-1)(2j-2)(2j-3)}\binom{i+j-3}{2j-4}\cos^{2j-1}(\frac{\pi}{2^n})\\\notag
&=\biggl((-1)^i (2i-1)+\sum_{j=1}^i(-1)^j2^{2j-1}\frac{2i-1}{2j-1}\binom{i+j-2}{2j-2}\biggr )\cos(\frac{\pi}{2^n})+2^{2i}\cos^{2i+1}(\frac{\pi}{2^n})\\\notag
&+\sum_{j=2}^{i}(-1)^{i+j-1}2^{2j-2}\frac{2+i-5i^2+2i^3-3j-4ij+4i^2j+j^2+2ij^2}{(2j-1)(i+j-1)(i+j-2)}\binom{i+j-1}{2j-2}\cos^{2j-1}(\frac{\pi}{2^n})\\\notag
&=\biggl((-1)^i (2i-1)+\sum_{j=1}^i(-1)^j2^{2j-1}\frac{2i-1}{2j-1}\binom{i+j-2}{2j-2}\biggr )\cos(\frac{\pi}{2^n})+2^{2i}\cos^{2i+1}(\frac{\pi}{2^n})\\\notag
&+\sum_{j=2}^{i}(-1)^{i+j-1}2^{2j-2}\frac{2i+1}{2j-1}\binom{i+j-1}{2j-2}\cos^{2j-1}(\frac{\pi}{2^n})\\\notag
&=\sum_{j=1}^{i+1}(-1)^{i+j-1}2^{2j-2}\frac{2i+1}{2j-1}\binom{i+j-1}{2j-2}\cos^{2j-1}(\frac{\pi}{2^n})\\\notag
\end{align}
The proof of the sine formula is similar. Instead of beginning with
\begin{align}\notag
&\cos\left(\frac{(2i-1)\pi}{2^n}+\frac{2\pi}{2^n}\right)=\cos(\frac{(2i-1)\pi}{2^n})\cos(\frac{2\pi}{2^n})-\sin(\frac{(2i-1)\pi}{2^n})\sin(\frac{2\pi}{2^n})\\\notag
\end {align}
we begin with
\begin{align}\notag
&\sin\left(\frac{(2i-1)\pi}{2^n}+\frac{2\pi}{2^n}\right)=\sin(\frac{(2i-1)\pi}{2^n})\cos(\frac{2\pi}{2^n})+\cos(\frac{(2i-1)\pi}{2^n})\sin(\frac{2\pi}{2^n})\\\notag
\end{align}
We convert all terms to sine, including using the binomial theorem, as above. We reorder terms, manipulate binomial coefficients and combine terms using the same two lemmas.
\end{proof}
\begin{proposition}
\begin{align}\notag
p_i(x)&=\frac{1}{2x}\biggl(\sqrt{x^2(x^2-1)}\left(\left(1-2x^2+2\sqrt{x^2(x^2-1)}\right)^i-\left(1-2x^2-2\sqrt{x^2(x^2-1)}\right)^i\right)\\\notag
&+x^2\left(\left(1-2x^2+2\sqrt{x^2(x^2-1)}\right)^i+\left(1-2x^2-2\sqrt{x^2(x^2-1)}\right)^i\right)\biggr)\\\notag
\end{align}

\end{proposition}
\begin{proof}
We note $p_1(x)=-x$, which agrees with the formula. We apply the Zeilberger algorithm to $p_i(x)$. The response is, if $i-1$ is a natural number, then:
$$-SUM[i]-2(2x^2-1)SUM[i+1]-SUM[i+2]=0$$
We then solve this recursion.
\end{proof}
This result can be generalised to a formula for the sine and cosine of odd multiples of an angle.
\begin{proposition}
Assume $i\in\mathbb{Z}$, $i\ge 0$, then
\begin{align}\notag
\sin((2i-1)\theta)&=-p_i(\sin(\theta))\\\notag
\cos((2i-1)\theta)&=(-1)^i p_i(\cos(\theta))\\\notag
\end{align}
\end{proposition}
\begin{proposition}
\begin{align}\notag
(-1)^ip_i\left(\cos\left(\frac{(2j-1)\pi}{2^n}\right)\right)&=\cos\left(\frac{\pi}{2^n}(2(2ij-j-i+1)-1)\right)\\\notag
-p_i\left(\sin\left(\frac{(2j-1)\pi}{2^n}\right)\right)&=\sin\left(\frac{\pi}{2^n}(2(2ij-j-i+1)-1)\right)\\\notag
\end{align}
\end{proposition}
\begin{proof}
In the previous proposition, set $\theta=\frac{(2j-1)\pi}{2^n}$
\end{proof}
\begin{corollary}
\begin{align}\notag
(-1)^ip_i((-1)^jp_j(x))=(-1)^jp_j((-1)^ip_i(x))\\\notag
\end{align}
\end{corollary}
\begin{proof}
\begin{align}\notag
(-1)^jp_j\left((-1)^ip_i\left(\cos\left(\frac{\pi}{2^n}\right)\right)\right)&=(-1)^jp_j\left(\cos\left(\frac{(2i-1)\pi}{2^n}\right)\right)=\cos\left(\frac{\pi}{2^n}(2(2ij-j-i+1)-1)\right)\\\notag
(-1)^ip_i\left((-1)^jp_j\left(\cos\left(\frac{\pi}{2^n}\right)\right)\right)&=(-1)^ip_i\left(\cos\left(\frac{(2j-1)\pi}{2^n}\right)\right)=\cos\left(\frac{\pi}{2^n}(2(2ij-j-i+1)-1)\right)\\\notag
\end{align}
We have two finite degree polynomials that agree on an infinite number of distinct values, and must therefore be equal.
\end{proof}
\begin{corollary}
\begin{align}\notag
 p_i(p_j(x))&=p_j(p_i(x))\\\notag
\end{align}
\end{corollary}
\begin{proof}
The functions $p_i$ are odd functions.
\end{proof}

\begin{proposition}
\begin{align}\notag
k&=i(2i-1)^{2^{n-1}-1}\\\notag
\end{align}
then
\begin{align}\notag
(-1)^kp_k\left(    (-1)^ip_i\left(\cos(\frac{(2j-1)\pi}{2^n})\right)   \right)&=\cos\left(\frac{(2j-1)\pi}{2^n}\right)\\\notag
\end{align}
\end{proposition}
\begin{proof}
We use Euler's theorem to solve for the cosine argument mod $2\pi$.
\end{proof}
\begin{corollary}
\begin{align}\notag
-p_{    i(2i-1)^{2^{n-1}-1}    }\left(\sin\left(\frac{(2i-1)\pi}{2^n}\right) \right)&=\sin\left(\frac{\pi}{2^n}\right)\\\notag
\end{align}
\end{corollary}
\begin{proof}
Use previous proposition.
\end{proof}
\begin{proposition}\label{permutation}
Regarding the $2^{n-2}x2^{n-2}$ matrices, $M_n$ which satisfy $$\cos^r(\frac{(2j-1)\pi}{2^n})=\frac{1}{2^{r-1}} \sum_{k=1}^{2^{n-2}}M_n(j,k)\cos(\frac{(2k-1)\pi}{2^n})$$ the set of entries of the first row are repeated in each subsequent row. Their position is permuted and their sign may change. The pattern of permutation and sign change is the same for all powers $r$.
\end{proposition}
\begin{proof}
Assume we have $$\cos^r(\frac{\pi}{2^n})=\frac{1}{2^{r-1}} \sum_{k=1}^{2^{n-2}}M_n(1,k)\cos(\frac{(2k-1)\pi}{2^n})$$ Under the automorphism $\sigma$ sending $\cos(\frac{\pi}{2^n})$ to $\cos(\frac{(2j-1)\pi}{2^n})$, the list of elements $\cos^j(\frac{(2k-1)\pi}{2^n})$, $1\le k\le 2^{n-1}-1$ is permuted. This because under the automorphism the image must still satisfy the minimal polynomial, $p(\sigma(x))=\sigma(p(x))$. The sign may change because the minimal polynomial has both signs for each root, but the final position and sign is independent of the power $r$. The permutation cannot result in the occurance of the same value with different sign because the sum of the two would then equal 0, which would be preserved by the automorphism, implying it was present in the first row. Note we can explicitly compute the permutation.  Each $\cos(\frac{(2k-1)\pi}{2^n})$ is expressable as a polynomial in $\cos(\frac{\pi}{2^n})$ (see the previous lemma). We can use these polynomials to compute the permutation.
\end{proof}
\begin{proposition}
We refer to notation in the previous lemma. Let
\begin{align}\notag
s&=\frac{2^n}{\pi}\arccos\biggl((-1)^jp_j\left(\cos\left(\frac{(2i-1)\pi}{2^n}\right)\right)\biggr)\\\notag
q&=\floor{\frac{s}{2^{n-1}}}\\\notag
m&=\biggl((-1)^q\frac{1}{2} \biggl(s \pmod*{2^{n-1}}+1\biggr)\biggr)\pmod*{2^{n-2}+1}\\\notag
\end{align}
 Regarding the $2^{n-2}x2^{n-2}$ matrices, $M_n$ which satisfy $$\cos^r(\frac{(2j-1)\pi}{2^n})=\frac{1}{2^{r-1}} \sum_{k=1}^{2^{n-2}}M_n(j,k)\cos(\frac{(2k-1)\pi}{2^n})$$ the set of entries of the first row are repeated in each subsequent row. The permutation of the position of the elements in in the first row which yields the position in the $i$th row is given by
$$M_n[i,m]=(-1)^qM_n[1,j]$$\notag
\end{proposition}
\begin{proof}
\begin{align}\notag
\cos\left(\frac{(2i-1)\pi}{2^n}\right) &=(-1)^i p_i\left(\cos\left(\frac{\pi}{2^n}\right)\right )\\\notag
\cos^r(\frac{\pi}{2^n})&=\frac{1}{2^{r-1}} \sum_{k=1}^{2^{n-2}}M_n(1,k)\cos(\frac{(2k-1)\pi}{2^n})\\\notag
\implies\\\notag
\cos^r\left(\frac{(2i-1)\pi}{2^n}\right) &=((-1)^ip_i(\cos(\frac{\pi}{2^n})))^r=\frac{1}{2^{r-1}}(-1)^i \sum_{j=1}^{2^{n-2}}M_n(1,j)p_i(\cos(\frac{(2j-1)\pi}{2^n}))\\\notag
&=\frac{1}{2^{r-1}}\sum_{k=1}^{2^{n-2}}M_n(i,k)\left(   \cos\left(\frac{(2k-1)\pi}{2^n}\right)  \right)\\\notag
\end{align}
We see that $M(1,j)$ is paired with $$(-1)^ip_i\left(\cos\left(\frac{(2j-1)\pi}{2^n}\right)\right)=\cos\left(\frac{(2(2ij-i-j+1)-1)\pi}{2^n}\right)$$ So, we want $M_n(i,m)$ where $1\le m\le 2^{n-2}$ satisfies 
\begin{align}\notag
\theta&=\frac{(2m-1)\pi}{2^n}=\arccos(\cos(\frac{ (2(2ij-i-j+1)-1)\pi}{2^n}))\\\notag
\end{align}
To follow the formulas, assume $\theta=\frac{(2k-1)\pi}{2^n}$. Then $s=2k-1$. Note that $0<m\le 2^{n-2}$. $s \pmod*{2^{n-1}}$ corresponds to $\theta\pmod*{\pi/2}$. Note that  to get the proper magnitude of $\cos(\theta)$ with $\theta$ replaced by a value $0<\frac{(2m-1)\pi}{2^n}<\pi/2$, we first solve for $\theta\pmod*{\pi/2}$.  $\bar{s}=s \pmod*{2^{n-1}}=(2k-1)\pmod*{2^{n-1}}$ will give a value $2m-1$ in the correct range. If $q=\floor{\theta/(\pi/2)}$ is even, we solve for $m=\frac{1}{2}(\bar{s}+1)$ and we are done. The outer mod then has no effect. If $q$ is odd, to get the correct magnitude cosine, we must use m=$(-\frac{1}{2}(\bar{s}+1))\pmod*{2^{n-2}+1}$. The ``$+1$'' at the end ensures we do not set $2^{n-2}$ equal to $0$. It doesn't cause problems due to the fact that the argument of the outer mod has magnitude $\le 2^{n-2}$. If $q$ is odd, we want
\begin{align}\notag
(2^{n-1}-(2k-1))\pi/2^n&=(2^{n-1}-2k+1)\pi/2^n\\\notag
\end{align}
We compute
\begin{align}\notag
(2(2^{n-2}+1-k)-1)\pi/2^n&=((2^{n-1}-2k+2)-1)\pi/2^n\\\notag
\end{align}
The arccos returns a value in the range $0\le\theta\le\pi$, so the sign of the cosine equals $(-1)^q$.

For example, let $n=5$, $i=3$ (3rd row of M).
\begin{center}
\begin{tabular}{|c c c c c c c c c |} 
 \hline
$j$ &$1$ &$2$&$3$ &$4$ &$5$ &$6$ &$7$ &$8$\\ 
 \hline
$m$&3&8&4&2&7&5&1&6 \\
 \hline
 sign &$+$ &$+$&$-$ &$-$ &$-$ &$+$ &$+$&$+$\\ 
 \hline
value &$M(1,7)$ &$-M(1,4)$&$M(1,1)$ &$-M(1,3)$ &$M(1,6)$ &$M(1,8)$ &$-M(1,5)$ &$M(1,2)$\\ 
 \hline
\end{tabular}
\end{center}
Note we are permuting the basis elements, but in the matrix representation the basis elements are fixed and we permute the coefficients. This implies that we apply the inverse permutation to the columns of the row. E.g., if we permute the 1st basis element into the 3rd basis element, the coefficient of the 3rd basis element is now what was previouly the coefficient in the 1st column. Hence, in our matrix representation, we permute the third column entry into the first column entry.
\end{proof}
Here is the pattern for 8x8 matricies corresponding to $n=5$:
\begin{align}\notag
\left(
\begin{array}{cccccccc}
  $M(1,1)$ &$M(1,2)$&$M(1,3)$ &$M(1,4)$ &$M(1,5)$ &$M(1,6)$ &$M(1,7)$ &$M(1,8)$ \\
 $-M(1,6)$ &$M(1,1)$&$-M(1,5)$ &$-M(1,7)$ &$M(1,2)$ &$-M(1,4)$ &$-M(1,8)$ &$M(1,3)$  \\
  $M(1,7)$ &$-M(1,4)$&$M(1,1)$ &$-M(1,3)$ &$M(1,6)$ &$M(1,8)$ &$-M(1,5)$ &$M(1,2)$ \\
  $M(1,5)$ &$-M(1,3)$&$-M(1,7)$ &$M(1,1)$ &$-M(1,8)$ &$-M(1,2)$ &$M(1,6)$ &$M(1,4)$ \\
  $M(1,4)$ &$-M(1,6)$&$-M(1,2)$ &$M(1,8)$ &$M(1,1)$ &$M(1,7)$ &$-M(1,3)$ &$-M(1,5)$ \\ 
  $-M(1,2)$ &$-M(1,5)$&$-M(1,8)$ &$M(1,6)$ &$M(1,3)$ &$M(1,1)$ &$M(1,4)$ &$M(1,7)$ \\
  $M(1,3)$ &$M(1,8)$&$-M(1,4)$ &$-M(1,2)$ &$-M(1,7)$ &$M(1,5)$ &$M(1,1)$ &$M(1,6)$ \\
  $-M(1,8)$ &$M(1,7)$&$-M(1,6)$ &$M(1,5)$ &$-M(1,4)$ &$M(1,3)$ &$-M(1,2)$ &$M(1,1)$ \\
\end{array}
\right)
\end{align}
Note, if we take the transpose of the above matrix and then interchange the row and column in the arguments of $M$, the result is also true:
\begin{align}\notag
\left(
\begin{array}{cccccccc}
  $M(1,1)$ &$-M(6,1)$&$M(7,1)$ &$M(5,1)$ &$M(4,1)$ &$-M(2,1)$ &$M(3,1)$ &$-M(8,1)$ \\
 $M(2,1)$ &$M(1,1)$&$-M(4,1)$ &$-M(3,1)$ &$-M(6,1)$ &$M(5,1)$ &$M(8,1)$ &$M(7,1)$  \\
  $M(3,1)$ &$-M(5,1)$&$M(1,1)$ &$-M(7,1)$ &$-M(2,1)$ &$-M(8,1)$ &$-M(4,1)$ &$-M(6,1)$ \\
  $M(4,1)$ &$-M(7,1)$&$-M(3,1)$ &$M(1,1)$ &$M(8,1)$ &$M(6,1)$ &$-M(2,1)$ &$M(5,1)$ \\
  $M(5,1)$ &$M(2,1)$&$M(6,1)$ &$-M(8,1)$ &$M(1,1)$ &$M(3,1)$ &$-M(7,1)$ &$-M(4,1)$ \\ 
  $M(6,1)$ &$-M(4,1)$&$M(8,1)$ &$-M(2,1)$ &$M(7,1)$ &$M(1,1)$ &$M(5,1)$ &$M(3,1)$ \\
  $M(7,1)$ &$-M(8,1)$&$-M(5,1)$ &$M(6,1)$ &$-M(3,1)$ &$M(4,1)$ &$M(1,1)$ &$-M(2,1)$ \\
  $M(8,1)$ &$M(3,1)$&$M(2,1)$ &$M(4,1)$ &$-M(5,1)$ &$M(7,1)$ &$M(6,1)$ &$M(1,1)$ \\
\end{array}
\right)
\end{align}
\begin{proposition}
Let
\begin{align}\notag
s&=\frac{2^n}{\pi}\arccos\biggl((-1)^jp_j\left(\cos\left(\frac{(2i-1)\pi}{2^n}\right)\right)\biggr)\\\notag
f(i,j)&=\frac{1}{2} \biggl(s \pmod*{2^{n-1}}+1\biggr)\biggr)\pmod*{2^{n-2}+1}\\\notag
\end{align}
Then the set of numbers $\{1,2,\ldots,2^{n-2}\}$ form a cyclic abelian group under the operation that maps $(a,b)\to f(a,b)$
\end{proposition}
\begin{proof}
Follows from Prop. 2.4 and Corollary 2.8.
\end{proof}
\begin{proposition}
Referring to the notation in the previous proposition, note that $m$ in prop. 2.12 equals $(-1)^q f(a,b)$. We define the operation $$a\circ b=(-1)^q f(a,b)$$ with, by definition, $(-a)\circ b=-(a\circ b)=a\circ(-b)$ for $1\le a,b\le 2^{n-2}$. Then
\begin{align}\notag
M(i,j&)=M(a\circ i,a \circ j)
\end{align}
\end{proposition}
\begin{proof}
Consider $m$ in prop. 2.12 as a function, $m=m(i,j)$. Unwinding the definition of $m$, if we assume  $m(i,q)=j$ (or, equivalently,  $i\circ q=j$), then $M(i,j)=M(1,q)$. We want to show $M(m(a,i),m(a,j))=M(i,j)$
\begin{align}\notag
M(m(a,i),m(a,j))&=M(a\circ i,a\circ j)=M(1,q)\\\notag
\end{align}
if $m(m(a,i),q)=m(a,j)$ but this is equivalent to $(a\circ i)\circ q=a\circ j$ so the result immediately follows from the group structure.
\end{proof}
\begin{corollary}
If we express the entries of $M$ in terms of the entries in the first row, as in the example above, then take the transpose and interchange the row and column entries, then we will end up with a representation in terms of the entries in the first column.
\end{corollary}
\begin{proof}
$A(j,i)=A(j\circ i^{-1},1)$ while $A(i,j)=A(1,i^{-1}\circ j)$
\end{proof}
\begin{lemma}
If $M(i,t)=M(1,r)$ and $M(j,t)=M(1,s)$ then there exists $k$ such that $M(k,j)=M(1,r)$ and $M(k,i)=M(1,s)$.
\end{lemma}
\begin{proof}
\begin{align}\notag
M(i,t)&=M(1,r)\implies r=i^{-1}\circ t\\\notag
M[j,t)&=M(1,s)\implies s=j^{-1}\circ t\\\notag
\end{align}

$k=j\circ r^{-1}$ satisfies $M(k,j)=M(1,r)$ and 
\begin{align}\notag
i^{-1}\circ t&=k^{-1}\circ j\implies \\\notag
(i\circ j^{-1})\circ i^{-1}\circ t&=(i\circ j^{-1})\circ k^{-1}j\implies\\\notag
j^{-1}\circ t&=k^{-1}\circ i \\\notag
\end{align}
\end{proof}
\begin{proposition}
$M$ is normal.
\end{proposition}
\begin{proof}
Let $A=M M^{t}$ and $B=M^{t}M$. Then $A(i,j)=\sum_{k=1}^{2^{n-2}}M(i,k)M(j,k)$ while  $B(i,j)=\sum_{k=1}^{2^{n-2}}M(k,i)M(k,j)$. But the previous lemma shows those two sums are equal.
\end{proof}
\begin{proposition}\label{commute}
Let $S$ be the set of $2^{n-2}x2^{n-2}$ matrices, $M_n$ which satisfy $$\cos^r(\frac{(2j-1)\pi}{2^n})=\frac{1}{2^{r-1}} \sum_{k=1}^{2^{n-2}}M_n(j,k)\cos(\frac{(2k-1)\pi}{2^n})$$ where $r$ is an odd integer and $n$ is a fixed integer.  Then if $A,B\in S$, $A B=B A$.
\end{proposition}
\begin{proof}
Let $C=A B$ and $D=B A$.
\begin{align}\notag
C(i,j)&=\sum_{k=1}^{2^{n-2}}A(i.k)B[k,j)\\\notag
D(i,j)&=\sum_{k=1}^{2^{n-2}}B(i.k)A[k,j)\\\notag
\end{align}
Looking at the indicies of $B$,  $B(k,j)=B(1,k^{-1}\circ j)$ and $B(i,k)=B(1,i^{-1}\circ k)$. But 
$$k^{-1}\circ j=i^{-1}\circ k\implies k^{-1}\circ i=j^{-1}\circ k$$
when we take the inverse of both sides. $$\implies A(i,k)=A(k,j)$$ Hence the correct $A$ entry is matched with the $B$ entry.
\end{proof}
\begin{proposition}
For $r\ {\rm odd},\ r\ge1,\ 2^{n-2}\ge\frac{r+1}{2}$
\begin{align}\notag
\cos^r\left(\frac{\pi}{2^n}\right)&=(-1)^i\frac{1}{2^{r-1}}\sum_{j=1}^{2^{n-2}}   \binom{r}{\frac{r-1}{2}-j+1}\cos\left(\frac{(2j-1)\pi}{2^n}                \right)
\end{align}
\end{proposition}
\begin{proof}
This is a special case of the next proposition.
\end{proof}
\begin{proposition}
For $n\ge2,\ r\ {\rm odd},\ r\ge1$
\begin{align}\notag
\cos^r\left(\frac{\pi}{2^n}\right)&=\frac{1}{2^{r-1}}\sum_{j=1}^{2^{n-2}}       \sum_{k=0}^{\floor{(r+1)/(2^n)}}\\\notag
&(-1)^k\left(              \binom{r}{\frac{r-1}{2}-(k 2^{n-1}+j-1)}-\binom{r}{\frac{r-1}{2}-((k+1)2^{n-1}-j)    }  \right)                \cos\left(\frac{(2j-1)\pi}{2^n}     \right)\\\notag
&=\frac{1}{2^{r-1}}\sum_{j=1}^{2^{n-2}}       \sum_{k=0}^{\infty}\\\notag
&(-1)^k\left(              \binom{r}{\frac{r-1}{2}-(k 2^{n-1}+j-1)}-\binom{r}{\frac{r-1}{2}-((k+1)2^{n-1}-j)    }  \right)                \cos\left(\frac{(2j-1)\pi}{2^n}     \right)\\\notag
\end{align}
\end{proposition}
\begin{proof}
The second equality follows because the binomial coefficient is $0$ for $k>\floor{(r+1)/(2^n)}$. That form makes the following induction argument clearer. The proof is by induction on the power $r$ to which the cosine is raised. Note the formula is true for $r=1$. There is a companion proposition involving the formula for even powers. The plan is to assume the odd power proposition is true for $r$ and to show the even power formula is then implied for $r+1$.  We will then show the even power formula for $r+1$ implies the odd power formula for $r+2$.
So, we assume
\begin{align}\notag
\cos^r\left(\frac{\pi}{2^n}\right)&=\frac{1}{2^{r-1}}\sum_{j=1}^{2^{n-2}}       \sum_{k=0}^{\floor{(r+1)/(2^n)}}\\\notag
&(-1)^k\left(              \binom{r}{\frac{r-1}{2}-(k 2^{n-1}+j-1)}-\binom{r}{\frac{r-1}{2}-((k+1)2^{n-1}-j)    }  \right)                \cos\left(\frac{(2j-1)\pi}{2^n}     \right)\\\notag
\end{align}
This implies
\begin{align}\notag
\cos^{r+1}\left(\frac{\pi}{2^n}\right)&=\frac{1}{2^{r-1}}\sum_{j=1}^{2^{n-2}}       \sum_{k=0}^{\floor{(r+1)/(2^n)}}(-1)^k\left(              \binom{r}{\frac{r-1}{2}-(k 2^{n-1}+j-1)}-\binom{r}{\frac{r-1}{2}-((k+1)2^{n-1}-j)    }  \right)  \\\notag
& \cos\left(\frac{(2j-1)\pi}{2^n}     \right)\cos^r\left(\frac{\pi}{2^n}\right)\\\notag
&=\frac{1}{2^{r-1}}\sum_{j=1}^{2^{n-2}}       \sum_{k=0}^{\floor{(r+1)/(2^n)}}(-1)^k\left(              \binom{r}{\frac{r-1}{2}-(k 2^{n-1}+j-1)}-\binom{r}{\frac{r-1}{2}-((k+1)2^{n-1}-j)    }  \right)  \\\notag
&  \frac{1}{2}   \biggl(         \cos\left(\frac{(2j-1)\pi}{2^n}-  \frac{\pi}{2^n}   \right)+ \cos\left(\frac{(2j-1)\pi}{2^n}+ \frac{\pi}{2^n}   \right)\biggr)\\\notag
&=\frac{1}{2^{r-1}}\sum_{j=1}^{2^{n-2}}       \sum_{k=0}^{\floor{(r+1)/(2^n)}}(-1)^k\left(              \binom{r}{\frac{r-1}{2}-(k 2^{n-1}+j-1)}-\binom{r}{\frac{r-1}{2}-((k+1)2^{n-1}-j)    }  \right)  \\\notag
&  \frac{1}{2}   \biggl(         \cos\left(\frac{(j-1)\pi}{2^{n-1}} \right)+ \cos\left(\frac{(j)\pi}{2^{n-1}}  \right)\biggr)\\\notag
&=\frac{1}{2^{r}}\biggl(    \sum_{k=0}^{\floor{(r+1)/(2^n)}}(-1)^k\left(              \binom{r}{\frac{r-1}{2}-(k 2^{n-1})}-\binom{r}{\frac{r-1}{2}-((k+1)2^{n-1}-1)    }  \right)  \\\notag
&+\sum_{j=1}^{2^{n-2}-1}       \sum_{k=0}^{\floor{(r+1)/(2^n)}}(-1)^k\left(              \binom{r}{\frac{r-1}{2}-(k 2^{n-1}+j-1)}-\binom{r}{\frac{r-1}{2}-((k+1)2^{n-1}-j)    }  \right)  \\\notag
&+ (-1)^k\left(              \binom{r}{\frac{r-1}{2}-(k 2^{n-1}+j)}-\binom{r}{\frac{r-1}{2}-((k+1)2^{n-1}-j-1)    }  \right) \biggr) \cos\left(\frac{(j)\pi}{2^{n-1}}  \right) \\\notag
&=\frac{1}{2^{r}}\biggl(    \sum_{k=0}^{\floor{(r+1)/(2^n)}}(-1)^k\left(              \binom{r}{\frac{r-1}{2}-(k 2^{n-1})}-\binom{r}{\frac{r-1}{2}-((k+1)2^{n-1}-1)    }  \right)  \\\notag
&+\sum_{j=1}^{2^{n-2}-1}       \sum_{k=0}^{\floor{(r+1)/(2^n)}}(-1)^k\biggl(            \left(  \binom{r}{\frac{r+1}{2}-(k 2^{n-1}+j)}+  \binom{r}{\frac{r+1}{2}-(k 2^{n-1}+j+1)}\right)\\\notag
&-\left(\binom{r}{\frac{r+1}{2}-((k+1)2^{n-1}-j+1)    } +\binom{r}{\frac{r-1}{2}-((k+1)2^{n-1}-j)    }  \right)\biggr) \biggr) \cos\left(\frac{(j)\pi}{2^{n-1}}  \right) \\\notag
&=\frac{1}{2^{r}}\biggl(       \sum_{k=0}^{\floor{(r+1)/(2^n)}}(-1)^k\biggl( \biggl(1+\frac{2(k2^{n-1}+1)}{r+1}\biggr)             \binom{r}{\frac{r+1}{2}-(k 2^{n-1}+1)}\\\notag
&-\biggl(1+\frac{2((k+1)2^{n-1})}{r+1}\biggr)\binom{r}{\frac{r+1}{2}-((k+1)2^{n-1})    }  \biggr)  \\\notag
&+\sum_{j=1}^{2^{n-2}-1}       \sum_{k=0}^{\floor{(r+1)/(2^n)}} (-1)^k\biggl(             \binom{r+1}{\frac{r+1}{2}-(k 2^{n-1}+j)}-  \binom{r+1}{\frac{r+1}{2}-((k+1) 2^{n-1}-j)}     \biggr)           \biggr) \cos\left(\frac{(j)\pi}{2^{n-1}}  \right) \\\notag
&=\frac{1}{2^{r}}\biggl(       \sum_{k=0}^{\floor{(r+1)/(2^n)}}(-1)^k\biggl(\biggl(  \left(\frac{r+1}{2}-k2^{n-1}\right)  /\left(\frac{r+1}{2}+k2^{n-1}+1\right)     \biggr)\\\notag
& \biggl(1+\frac{2(k2^{n-1}+1)}{r+1}\biggr)             \binom{r}{\frac{r+1}{2}-(k 2^{n-1}+1)}-\biggl(1+\frac{2((k+1)2^{n-1})}{r+1}\biggr)\binom{r}{\frac{r+1}{2}-((k+1)2^{n-1})    }  \biggr)  \\\notag
&+\sum_{j=1}^{2^{n-2}-1}       \sum_{k=0}^{\floor{(r+1)/(2^n)}} (-1)^k\biggl(             \binom{r+1}{\frac{r+1}{2}-(k 2^{n-1}+j)}-  \binom{r+1}{\frac{r+1}{2}-((k+1) 2^{n-1}-j)}     \biggr)           \biggr) \cos\left(\frac{(j)\pi}{2^{n-1}}  \right) \\\notag
\end{align}
\begin{align}
&=\frac{1}{2^{r}}\biggl(       \sum_{k=0}^{\floor{(r+1)/(2^n)}}(-1)^k\biggl(\frac{1+r-k2^{n}}{1+r}  \binom{r}{\frac{r+1}{2}-(k 2^{n-1}+1)}-\frac{1+r+(1+k)2^n}{1+r}\binom{r}{\frac{r+1}{2}-((k+1)2^{n-1})    }  \biggr)  \\\notag
&+\sum_{j=1}^{2^{n-2}-1}       \sum_{k=0}^{\floor{(r+1)/(2^n)}} (-1)^k\biggl(             \binom{r+1}{\frac{r+1}{2}-(k 2^{n-1}+j)}-  \binom{r+1}{\frac{r+1}{2}-((k+1) 2^{n-1}-j)}     \biggr)           \biggr) \cos\left(\frac{(j)\pi}{2^{n-1}}  \right) \\\notag
\end{align}
Comparing with the formula for even  $r$, we see that to complete the proof we need to show
\begin{align}
 &\sum_{k=0}^{\floor{(r+1)/(2^n)}}(-1)^k\biggl(\frac{1+r-k2^{n}}{1+r}  \binom{r}{\frac{r+1}{2}-(k 2^{n-1}+1)}-\frac{1+r+(1+k)2^n}{1+r}\binom{r}{\frac{r+1}{2}-((k+1)2^{n-1})    }  \biggr)  \\\notag
 &=   \sum_{k=0}^{\floor{(r+1)/(2^n)}}(-1)^k\biggl( \binom{r}{\frac{r+1}{2}-(k 2^{n-1}+1)}-\binom{r}{\frac{r+1}{2}-((k+1)2^{n-1})    }  \biggr)  \\\notag
\end{align}
Or that
\begin{align}
 &\sum_{k=0}^{\floor{(r+1)/(2^n)}}(-1)^k\biggl(\frac{-k2^{n}}{1+r}  \binom{r}{\frac{r+1}{2}-(k 2^{n-1}+1)}-\frac{(1+k)2^n}{1+r}\binom{r}{\frac{r+1}{2}-((k+1)2^{n-1})    }  \biggr) =0 \\\notag
\end{align}
but this is a telescoping sum with first and last terms $0$.
\end{proof}
\begin{lemma}
Let
\begin{align}\notag
m&=\left((-1)^{\floor{(i-1)/2^{n-2}}} (i-2^{n-2}\floor{(i-1)/2^{n-2})}\right)\pmod*{2^{n-2}+1}\\\notag
\end{align}
Then $0\le m\le 2^{n-2}$ and
\begin{align}\notag
\cos\left(\frac{(2i-1)\pi}{2^n}\right) &=(-1)^{\floor{(2^{n-2}+i-1)/2^{n-1}}}\cos\left(\frac{(2m-1)\pi}{2^n}\right) \\\notag
\end{align}
\end{lemma}
\begin{proof}
The proof consists of checking the equality when $\theta=\frac{(2i-1)\pi}{2^n}$ lies in each quadrant of the polar coordinate plane.
\end{proof}
\begin{proposition}\label{oddpowers}
For $n\ge2,\ r\ {\rm odd},\ r\ge1$
\begin{align}\notag
\cos^r\left(\frac{(2i-1)\pi}{2^n}\right)&=\frac{1}{2^{r-1}}\sum_{j=1}^{2^{n-2}}  (-1)^{\floor{(2^{n-2}+2ij-i-j)/2^{n-1}}}     \sum_{k=0}^{\floor{(r+1)/(2^n)}}\\\notag
&(-1)^k\left(              \binom{r}{\frac{r-1}{2}-(k 2^{n-1}+j-1)}-\binom{r}{\frac{r-1}{2}-((k+1)2^{n-1}-j)    }  \right)       \\\notag
&\cos(      \frac{   (  2[\biggl(   (-1)^{\floor{ (2ij-i-j)/2^{n-2} } } ( (2 i j-i-j)\pmod*{2^{n-2}}+1)\biggr)\pmod*{2^{n-2}+1}   ]  -1   )\pi   }    {2^{n}}      )\
\end{align}
\begin{proof}
From proposition \ref{permutation} we see that the proof amounts to showing the proper permutation has been applied to the previous proposition.
\begin{align}
\cos^r\left(\frac{(2i-1)\pi}{2^n}\right)&=\left( (-1)^i p_i\left(   \cos\left(\frac{\pi}{2^n}\right)  \right)\right)^r\\\notag
&=\frac{1}{2^{r-1}}\sum_{j=1}^{2^{n-2}}       \sum_{k=0}^{\floor{(r+1)/(2^n)}}(-1)^k\left(              \binom{r}{\frac{r-1}{2}-(k 2^{n-1}+j-1)}-\binom{r}{\frac{r-1}{2}-((k+1)2^{n-1}-j)    }  \right) \\\notag
&     (-1)^i p_i\left(    \cos\left (\frac{(2j-1)\pi}{2^n} \right)  \right)\\\notag
&=\frac{1}{2^{r-1}}\sum_{j=1}^{2^{n-2}}       \sum_{k=0}^{\floor{(r+1)/(2^n)}}(-1)^k\left(              \binom{r}{\frac{r-1}{2}-(k 2^{n-1}+j-1)}-\binom{r}{\frac{r-1}{2}-((k+1)2^{n-1}-j)    }  \right) \\\notag
&   \cos\left (\frac{(2(2ij-i-j+1)-1)\pi}{2^n} \right)\\\notag
\end{align}
Now apply the preceding lemma.
\end{proof}
\end{proposition}
\begin{proposition}
Assume $n\ge2,\ r\ {\rm odd},\ r\ge1$. Regarding the $2^{n-2}x2^{n-2}$ matrices, $M_n$ which satisfy $$\cos^r(\frac{(2j-1)\pi}{2^n})=\frac{1}{2^{r-1}} \sum_{k=1}^{2^{n-2}}M_n(j,k)\cos(\frac{(2k-1)\pi}{2^n})$$
Let
\begin{align}\notag
p &=2ij-i-j+1\\\notag
s &= \floor{\frac{(p - 1)}{2^{n-2}}}\\\notag
m &= ((-1)^s (p - s 2^{n-2}))\pmod*{2^{n-2}+1}\\\notag
q&=\floor{    (2^{n-2}+2 i j-i-j)/(2^{n-1})   }
\end{align}
then
\begin{align}\notag
M_n[i,m]&= (-1)^{q}\left(       \sum_{k=0}^{\floor{  (r+1)/(2^n)  } }(-1)^k\left(              \binom{r}{\frac{r-1}{2}-(k2^{n-1}+j-1)}-\binom{r}{\frac{r-1}{2}-((k+1)2^{n-1}-j)   }  \right)     \right)  \\\notag
\end{align}
\end{proposition}
\begin{proof}
This is a restatement of the previous proposition. Note
\begin{align}\notag
p-s 2^{n-2}&=(p-1)\pmod*{2^{n-2}}+1\\\notag
\end{align}
\end{proof}

\begin{corollary} 
Let $S_n=\sum_{i=1}^{2^{n-2}}\cos^{r}\left(\frac{(2i-1)\pi}{2^n}\right)$, $r$ odd, $r\ge 1$,
\begin{align}\notag
S_n&=\frac{1}{2^{r-1}} \sum_{i=1}^{2^{n-2}} \sum_{j=1}^{2^{n-2}} (-1)^{\floor{    (2^{n-2}+2 i j-i-j)/(2^{n-1})   } }\sum_{k=0}^{\floor{  (r+1)/(4n)   }}\\\notag
&(-1)^k \left(              \binom{r}{\frac{r-1}{2}-(2k n+j-1)}-\binom{r}{\frac{r-1}{2}-(2(k+1)n-j)   }  \right)\\\notag
&\cos(      \frac{   (  2[  \left( (-1)^{\floor{ (2ij-i-j)/2^{n-2} } } ( (2 i j-i-j)\pmod*{2^{n-2}}+1)\right)\pmod*{2^{n-2}+1}   ]  -1   )\pi   }    {2^{n}}      )\\\notag
&=\frac{1}{2^{r-1}}\sum_{j=1}^{2^{n-2}}\sum_{k=0}^{\floor{  (r+1)/(4n)  } }(-1)^k \left(              \binom{r}{\frac{r-1}{2}-(2k n+j-1)}-\binom{r}{\frac{r-1}{2}-(2(k+1)n-j)   }  \right)\\\notag
& \sum_{i=1}^{2^{n-2}} (-1)^{\floor{    (2^{n-2}+2 i j-i-j)/(2^{n-1})   }}\\\notag
&\cos(      \frac{   (  2[   (-1)^{\floor{ (2ij-i-j)/2^{n-2} } } ( (2 i j-i-j)\pmod*{2^{n-2}}+1)\pmod*{2^{n-2}+1}   ]  -1   )\pi   }    {2^{n}}      )\\\notag
\end{align}
\end{corollary}

The Mathematica code for producing the matrix for terms of $M$, $S_n$ for $r>1$, $r$ odd is given in the appendix.
\begin{lemma}
\begin{align}\notag
&\binom{r}{\frac{r-1}{2}-(k2^{n-1}+j-1)}-\binom{r}{\frac{r-1}{2}-((k+1)2^{n-1}-j)   }\\\notag
&=-\left(  \binom{r}{\frac{r-1}{2}-(k2^{n-1}+(2^{n-1}-j+1)-1)}-\binom{r}{ \frac{r-1}{2}-( (k+1)2^{n-1}-(2^{n-1}-j+1) )   }    \right)\\\notag
\end{align}
\end{lemma}
\begin{lemma}
In the previous proposition, the value $(-1)^q$ is a function of $i,j,n$. Assume we denote that function $e(i,j,n)$. Then
$$e(i,j,n)=-e(i,2^{n-1}-j+1,n)$$
and
$$e(i,((i+j-1)(2i-1)^{2^{n-2}-1})\pmod*{2^{n-1}},n)=(-1)^{\floor{(i+j-1}(2i-1)^{2^{n-2}-1}/2^{n-1}}$$
\end{lemma}
\begin{proof}
For $1\le i,j\le 2^{n-2}$ we want to show
\begin{align}\notag
(-1)^{\floor{    (2^{n-2}+2 i j-i-j)/2^{n-1}  }}&=-(-1)^{\floor{    (2^{n-2}+2 i (2^{n-1}-j+1)-i-(2^{n-1}-j+1))/2^{n-1}   }}\\\notag
\end{align}
or, equivalently,
\begin{align}\notag
(-1)^{\floor{    (2^{n-2}+(2i-1)j-i)/2^{n-1}   }}&=-(-1)^{\floor{    (2^{n-2}+(2 i-1) (2^{n-1}-j)+i-1))/2^{n-1}   }}\\\notag
\end{align}
First note that, setting $j=1$, we get $(-1)^0=-(-1)^{2i-1}$, verifying the statement for $j=1$.
Next note that when we add the arguments of the floor function, we get
\begin{align}\notag
( (2^{n-2}+(2i-1)j-i)+    2^{n-2}+(2 i-1) (2^{n-1}-j)+i-1))/2^{n-1}  &=(i 2^n-1)/2^{n-1} \\\notag
\end{align}
By induction, assume true for $j$. When $j$ increases by 1, the argument of the first floor function increases by at most one, while the argument of the second floor function decreases by at most 1. Now, assume we start with $j=j$ and we increase $j$. The value of the first floor function will increase by 1 when the numerator reaches the value $m2^{n-1}$, for some $m$. But at that point the numerator of the second floor function decreases to $(i-m)2^{n-1}-1$, exactly when its floor function decreases by 1.
\end{proof}
\begin{lemma}
In the previous proposition, the value $m$ is a function of $i,j,n$. Assume we denote that function $m(i,j,n)$. Then
$$m(i,j,n)=m(i,2^{n-1}-j+1,n)$$
\end{lemma}
\begin{proof}
We want to show
\begin{align}\notag
&\biggl(    (-1)^{   \floor{((2i-1)j-i)/2^{n-2}}  }   \left( ((2i-1)j-i)\pmod*{2^{n-2}} +1 \right) \biggr)\pmod*{2^{n-2}+1}\\\notag
&=\biggl(    (-1)^{   \floor{((2i-1)(2^{n-1}-j+1)-i)/2^{n-2}}  }   \left( ((2i-1)(2^{n-1}-j+1)-i)\pmod*{2^{n-2}} +1 \right) \biggr)\pmod*{2^{n-2}+1}\\\notag
\end{align}
First, foocusing on the factors of $(-1)$, we claim
$$  (-1)^{   \floor{((2i-1)j-i)/2^{n-2}}  }=-  (-1)^{   \floor{((2i-1)(2^{n-1}-j+1)-i)/2^{n-2}}  } $$
Note when $j=1$ we have $(-1)^0=-(-1)^{4i-3}$, confirming for $j=1$. The sum of the arguements to the floor function equals $((2i-1)2^{n-1}-1)/2^{n-2}$. The floor function on the left increases by 1 when the numerator of the floor function on the left reaches $m 2^{n-2}$, the numerator of the floor function on the right equals $((2i-1)2^{n-1}-1)-m 2^{n-2}$, which is when that floor function deceases by 1, demonstrating the claim. Likewise, the sum of the arguments to the inner mod functions equals $((2i-1)2^{n-1}-1)$. If we call the left inner mod function argument $a$ and the right $b$, then $b=-1 -a \pmod*{ 2^{n-2}}$, or, $b=k2^{n-2}-1-a$, for some $k$. This implies $-(b+1)=-k 2^{n-2}+a+1$, which confirms the lemma.
\end{proof}
\begin{proposition}
In the previous proposition, the value $M[i,m]$ is a function of $i,j,n$. Assume we denote that function $f(i,j,n)$. Then
$$f(i,j,n)=f(i,2^{n-1}-j+1,n)$$
\end{proposition}
\begin{proof}
The proposition is a consequence of applying the three previous lemmas to the previous proposition.
\end{proof}

\begin{corollary}
Regarding the $2^{n-2}x2^{n-2}$ matrices, $M_n$ which satisfy $$\cos^r(\frac{(2j-1)\pi}{2^n})=\frac{1}{2^{r-1}} \sum_{k=1}^{2^{n-2}}M_n(j,k)\cos(\frac{(2k-1)\pi}{2^n})$$
Let
\begin{align}\notag
p &=(i+j-1)(2i-1)^{2^{n-2}-1}\mod{2^{n-1}}\\\notag
q&=\floor{(i+j-1)(2i-1)^{2^{n-2}-1}/(2^{n-1})   }\\\notag
\end{align}
then
\begin{align}\notag
M_n[i,j]&= (-1)^{q}\left(       \sum_{k=0}^{\floor{  (r+1)/(2^n)  } }(-1)^k\left(              \binom{r}{\frac{r-1}{2}-(k 2^{n-1}+p-1)}-\binom{r}{\frac{r-1}{2}-((k+1)2^{n-1}-p)   }  \right)     \right)  \\\notag
\end{align}
\begin{proof}
We need to show that, taking the $p$ from this corollary and substuting it for $j$ in the previous propositon, the result is $m=j$.
\begin{align}\notag
m&= ((-1)^{\floor{\frac{(2ij-i-j+1 - 1)}{2^{n-2}}}} (2ij-i-j+1 - \floor{\frac{(2ij-i-j+1 - 1)}{2^{n-2}}}2^{n-2}))\pmod*{2^{n-2}+1}\\\notag
&=((-1)^{\floor{\frac{(2ij-i-j+1 - 1)}{2^{n-2}}}} (2ij-i-j+1)\pmod*{2^{n-2}}+1))\pmod*{2^{n-2}+1}\\\notag
\end{align}
substitting $(i+j-1)(2i-1)^{2^{n-2}-1}\pmod*{2^{n-1}}$ for$j$ in $2ij-i-j$ yields:
\begin{align}\notag
&(2i-1)((i+j-1)(2i-1)^{2^{n-2}-1}\pmod*{2^{n-1}})-i\\\notag
&=i+j-1+k 2^{n-1}-i\\\notag
&=j-1+k2^{n-1}\\\notag
\end{align}
Hence
\begin{align}\notag
m&= (   (-1)^{  \floor{    \frac{j-1+k2^{n-1}}{2^{n-2}} }   }   ((j-1+k2^{n-1})    \pmod*{2^{n-2}+1} )   )\pmod*{2^{n-2}+1}\\\notag
&= ((-1)^{\floor{    \frac{j-1}{2^{n-2}}   }   }   j )  \pmod*{2^{n-2}+1}\\\notag
&=j\\\notag
\end{align}
However, $1\le p\le 2^{n-1}$ whereas in the proposition, $1\le j\le 2^{n-2}$. If $p>2^{n-2}$, then, by the previous proposition, $f(i,p,n)=f(i,2^{n-1}-p+1,n)$.
\end{proof}
\end{corollary}
\begin{corollary}
\begin{align}\notag
\sum_{i=1}^{2^{n-2}}\cos^{r}\left(\frac{(2i-1)\pi}{2^n}\right)&=\sum_{j=1}^{2^{n-2}}c_j \cos\left(\frac{(2j-1)\pi}{2^n}\right) \\\notag
\end{align}
where
\begin{align}\notag
c_j&=\sum_{i=1}^{2^{n-2}} (-1)^{\floor{(i+j-1)(2i-1)^{2^{n-2}-1}/(2^{n-1})   }}     \sum_{k=0}^{\floor{  (r+1)/(2^n)  } }\\\notag
&(-1)^k\Biggl(     \binom{r}{\frac{r-1}{2}-(k 2^{n-1}+\left((i+j-1)(2i-1)^{2^{n-2}-1}\pmod*{2^{n-1}}\right)-1)} \\\notag
&   -\binom{r}{\frac{r-1}{2}-((k+1)2^{n-1}-\left((i+j-1)(2i-1)^{2^{n-2}-1}\pmod*{2^{n-1}}\right)   } \Biggr)\\\notag
\end{align}
\end{corollary}
\begin{corollary}
\begin{align}\notag
\sum_{i=1}^{2^{m-1}-1}\cos^{r}\left(\frac{i\pi}{2^m}\right)&=\frac{1}{2^{r-1}}\sum_{n=2}^m\sum_{j=1}^{2^{n-2}}c(j,n) \cos\left(\frac{(2j-1)\pi}{2^n}\right) \\\notag
\end{align}
where
\begin{align}\notag
c(j,n)&=\sum_{i=1}^{2^{n-2}} (-1)^{\floor{(i+j-1)(2i-1)^{2^{n-2}-1}/(2^{n-1})   }}     \sum_{k=0}^{\floor{  (r+1)/(2^n)  } }\\\notag
&(-1)^k\Biggl(     \binom{r}{\frac{r-1}{2}-(k 2^{n-1}+\left((i+j-1)(2i-1)^{2^{n-2}-1}\pmod*{2^{n-1}}\right)-1)} \\\notag
&   -\binom{r}{\frac{r-1}{2}-((k+1)2^{n-1}-\left((i+j-1)(2i-1)^{2^{n-2}-1}\pmod*{2^{n-1}}\right)   } \Biggr)\\\notag
\end{align}
\end{corollary}
For another example, for $n=4,\ r=7$
\begin{equation}\notag
\left(
\begin{array}{cc}
\cos^{7}\left(\frac{\pi}{16}\right)   \\
\cos^{7}\left(\frac{3\pi}{16}\right)  \\
\cos^{7}\left(\frac{5\pi}{16}\right)  \\
\cos^{}\left(\frac{7\pi}{16}\right)  \\
\end{array}
\right)=\frac{1}{2^{6}}\left(
\begin{array}{cccc}
35 & 21 & 7 & 1 \\
 -7 & 35 & -1 & -21 \\
 -21 & 1& 35 & 7 \\
 -1 & 7 & -21 & 35 \\
\end{array}
\right)\left(
\begin{array}{cc}
\cos\left(\frac{\pi}{16}\right)   \\
\cos\left(\frac{3\pi}{16}\right)  \\
\cos\left(\frac{5\pi}{16}\right)  \\
\cos\left(\frac{7\pi}{16}\right)  \\
\end{array}
\right)
\end{equation}
We know that this matrix $M$ is normal. Let
\begin{equation}\notag
P=
\left(
\begin{array}{cccc}
i & -i & -i & i \\
 \frac{\sqrt{2}}{2}(-1+i) &  -\frac{\sqrt{2}}{2}(1+i) &  \frac{\sqrt{2}}{2}(1+i) &  \frac{\sqrt{2}}{2}(1-i) \\
 -\frac{\sqrt{2}}{2}(1+i) &  \frac{\sqrt{2}}{2}(-1+i) &  \frac{\sqrt{2}}{2}(1-i) &  \frac{\sqrt{2}}{2}(1+i) \\
 1 & 1 & 1 & 1 \\
\end{array}
\right)
\end{equation}
The columns of $P$ are the orthonormal eigenvectors of $M$ and $D=P M P^{-1}$ is a diagonal marix. If one considers a matrix
\begin{equation}\notag
N=
\left(
\begin{array}{cccc}
a &b & c &d \\
-c &  a & -d & -b \\
 -b & d&  a &  c \\
 -d & c & -b & a \\
\end{array}
\right)
\end{equation}
We compute
\begin{equation}\notag
P N P^{-1}=
\left(
\begin{array}{cccc}
a+\frac{1+i}{\sqrt{2}}b-\frac{1-i}{\sqrt{2}}c-i d &0 & 0 &0\\
0&  a+\frac{1-i}{\sqrt{2}}b-\frac{1+i}{\sqrt{2}}c+i d & 0 & 0 \\
 0 & 0&  a-\frac{1-i}{\sqrt{2}}b+\frac{1+i}{\sqrt{2}}c+i d &  0 \\
0 & 0 & 0 & a-\frac{1+i}{\sqrt{2}}b+\frac{1-i}{\sqrt{2}}c-i d \\
\end{array}
\right)
\end{equation}
So, the eigenvectors apply to all matricies of this form. Hence, these matricies are simultaneously diagonizable and commute. Hence the product of two matricies of this form yields a matrix of this form. Matricies of this form constitute an algebra. If in $N$ above, we successively set one of $a,b,c,d$ equal to $\pm 1$ and other values to $0$, we get a representation of $\mathbb{Z}/8\mathbb{Z}$. This pattern holds for negative odd powers where $b,d\le0$.
\section{Even Powers Of Cosine}
For $n\ge 0,\ n,r\in\mathbb{Z}$, $r$ even, there is a $2^{n-2}x2^{n-2}$ matrix $M$ such that
\begin{equation}\notag
\left(
\begin{array}{cc}
\cos^r\left(\frac{\pi}{2^n}\right)   \\
\cos^r\left(\frac{3\pi}{2^n}\right)  \\
\vdots\\
\cos^r\left(\frac{(2^{n-2}-1)\pi}{2^n}\right)  \\
\end{array}
\right)=\frac{1}{2^{r-1}}M\left(
\begin{array}{cc}
1\\
\cos\left(\frac{\pi}{2^{n-1}}\right)   \\
\cos\left(\frac{2\pi}{2^{n-1}}\right)  \\
\vdots\\
\cos\left(\frac{(2^{n-2}-1)\pi}{2^{n-1}}\right)  \\
\end{array}
\right)
\end{equation}

For example, for $n=4,\ r=16$
\begin{equation}\notag
\left(
\begin{array}{cc}
\cos^{16}\left(\frac{\pi}{16}\right)   \\
\cos^{16}\left(\frac{3\pi}{16}\right)  \\
\cos^{16}\left(\frac{5\pi}{16}\right)  \\
\cos^{16}\left(\frac{7\pi}{16}\right)  \\
\end{array}
\right)=\frac{1}{2^{15}}\left(
\begin{array}{cccc}
6434 &11424 & 7888 & 3808 \\
6434 & -3808 & -7888 & 11424 \\
 6434 & 3808 & -7888 & -11424 \\
 6434 & -11424& 7888 & -3808 \\
\end{array}
\right)\left(
\begin{array}{cc}
1 \\
\cos\left(\frac{\pi}{8}\right)   \\
\cos\left(\frac{\pi}{4}\right)  \\
\cos\left(\frac{3\pi}{8}\right)  \\
\end{array}
\right)
\end{equation}
While  for $n=5,\ r=16$
\begin{align}\notag
&\left(
\begin{array}{cc}
\cos^{16}\left(\frac{\pi}{32}\right)   \\
\cos^{16}\left(\frac{3\pi}{32}\right)  \\
\cos^{16}\left(\frac{5\pi}{32}\right)  \\
\cos^{16}\left(\frac{7\pi}{32}\right)  \\
\cos^{16}\left(\frac{9\pi}{32}\right)   \\
\cos^{16}\left(\frac{11\pi}{32}\right)  \\
\cos^{16}\left(\frac{13\pi}{32}\right)  \\
\cos^{16}\left(\frac{15\pi}{32}\right)  \\
\end{array}
\right)\\\notag
&=\frac{1}{2^{15}}\left(
\begin{array}{cccccccc}
6435 & 11440 & 8008& 4368& 1820& 560& 120& 16 \\
6435 & -560 & -120& 11440& -1820& -16& 8008& -4368 \\
6435 & -4368 & 120& 16& -1820& 11440& -8008& 560 \\
6435 & -16 & -8008& 560& 1820& -4368& -120& 11440 \\
6435 & 16 & -8008& -560& 1820& 4368& -120& -11440 \\
6435 & 4368 & 120& -16& -1820& -114400& -8008& -560 \\
6435 & 560 & -120& -11440& -1820& 16& 8008& 4368\\
6435 & -11440 & 8008& -4368& 1820& -560& 120& -16 \\
\end{array}
\right)
\left(
\begin{array}{cc}
1 \\
\cos\left(\frac{\pi}{16}\right)   \\
\cos\left(\frac{\pi}{8}\right)  \\
\cos\left(\frac{3\pi}{16}\right)  \\
\cos\left(\frac{\pi}{4}\right)  \\
\cos\left(\frac{5\pi}{16}\right)   \\
\cos\left(\frac{3\pi}{8}\right)  \\
\cos\left(\frac{7\pi}{16}\right)  \\
\end{array}
\right)
\end{align}
Note that the above matrix is comprised of sub-matrices corresponding to cosine arguments of the form $\frac{(2i-1)\pi}{2^m}$.
$$m=3:$$
\begin{align}\notag
\left(
\begin{array}{cc}
8008 & 120 \\
-120 & 8008 \\
\end{array}
\right),
\left(
\begin{array}{cc}
120 & -8008\\
-8008 & -120 \\
\end{array}
\right),
\left(
\begin{array}{cc}
-8008 & -120 \\
120 & -8008 \\
\end{array}
\right),
\left(
\begin{array}{cc}
-120 & 8008\\
8008 & 120 \\
\end{array}
\right)
\end{align}
$$m=4:$$
\begin{align}\notag
\left(
\begin{array}{cccc}
 11440 & 4368& 560& 16 \\
-560 & 11440& -16& -4368 \\
 -4368 &16& 11440& 560 \\
 -16 &560& -4368&  11440 \\
\end{array}
\right),
\left(
\begin{array}{cccc}
16 & -560&4368&-11440 \\
4368 & -16& -114400&-560 \\
560 & -11440& 16& 4368\\
-11440 & -4368& -560& -16 \\
\end{array}
\right)
\end{align}
In the following, we assume  $ n,r\in\mathbb{Z}$.
\begin{proposition}
For $r\ {\rm even},\ r\ge1,\ 2^{n-2}\ge\frac{r+1}{2}$
\begin{align}\notag
\cos^r\left(\frac{\pi}{2^n}\right)&=\frac{1}{2^{r-1}}\left(\frac{1}{2}\binom{r}{\frac{r}{2}}+\sum_{j=1}^{2^{n-2}-1}   \binom{r}{\frac{r}{2}-j}\cos\left(\frac{j\pi}{2^{n-1}}                \right)\right)
\end{align}

\end{proposition}

\begin{proposition}
For $n\ge2,\ r\ {\rm even},\ r\ge 0$
\begin{align}\notag
\cos^r\left(\frac{\pi}{2^n}\right)&=\frac{1}{2^{r-1}}\Biggl(\frac{1}{2}\sum_{k=0}^{\floor{r/2^n}}(-1)^k\left(\binom{r}{\frac{r}{2}-k2^{n-1}}-\binom{r}{\frac{r}{2}-(k+1)2^{n-1}}\right)\\\notag\\\notag
&+\sum_{j=1}^{2^{n-2}-1}     \sum_{k=0}^{\floor{r/2^n}}\\\notag
&(-1)^k\left(              \binom{r}{\frac{r}{2}-(k 2^{n-1}+j)}-\binom{r}{\frac{r}{2}-((k+1)2^{n-1}-j)    }  \right)                \cos\left(\frac{j\pi}{2^{n-1}}     \right)\Biggr)\\\notag
&=\frac{1}{2^{r-1}}\Biggl(\frac{1}{2}\sum_{k=0}^{\infty}(-1)^k\left(\binom{r}{\frac{r}{2}-k2^{n-1}}-\binom{r}{\frac{r}{2}-(k+1)2^{n-1}}\right)\\\notag\\\notag
&+\sum_{j=1}^{2^{n-2}-1}     \sum_{k=0}^{\infty}\\\notag
&(-1)^k\left(              \binom{r}{\frac{r}{2}-(k 2^{n-1}+j)}-\binom{r}{\frac{r}{2}-((k+1)2^{n-1}-j)    }  \right)                \cos\left(\frac{j\pi}{2^{n-1}}     \right)\Biggr)\\\notag
\end{align}
\end{proposition}
\begin{proof}
Note the second inequalty is true because the value of the binomial is $0$ for $k>\floor{r/2^n}$. This form is more convenient for the following induction argument. The proof of this proposition is in conjunction with the proof of the corresponding proposition for odd powers $r$. The proof procedes by induction. The proposition is true for $r=0$. We assume true for $r$ and show it implies the formula for odd power $r+1$.  So, assume
\begin{align}\notag
\cos^r\left(\frac{\pi}{2^n}\right)&=\frac{1}{2^{r-1}}\Biggl(\frac{1}{2}\sum_{k=0}^{\floor{r/2^n}}(-1)^k\left(\binom{r}{\frac{r}{2}-k2^{n-1}}-\binom{r}{\frac{r}{2}-(k+1)2^{n-1}}\right)\\\notag\\\notag
&+\sum_{j=1}^{2^{n-2}-1}     \sum_{k=0}^{\floor{r/2^n}}\\\notag
&(-1)^k\left(              \binom{r}{\frac{r}{2}-(k 2^{n-1}+j)}-\binom{r}{\frac{r}{2}-((k+1)2^{n-1}-j)    }  \right)                \cos\left(\frac{j\pi}{2^{n-1}}     \right)\Biggr)\\\notag
\end{align}
This implies
\begin{align}\notag
\cos^{r+1}\left(\frac{\pi}{2^n}\right)&=\frac{1}{2^{r-1}}\Biggl(\frac{1}{2}\sum_{k=0}^{\floor{r/2^n}}(-1)^k\left(\binom{r}{\frac{r}{2}-k2^{n-1}}-\binom{r}{\frac{r}{2}-(k+1)2^{n-1}}\right)\cos\left(\frac{\pi}{2^n}\right)\\\notag\\\notag
&+\sum_{j=1}^{2^{n-2}-1}     \sum_{k=0}^{\floor{r/2^n}}\\\notag
&(-1)^k\left(              \binom{r}{\frac{r}{2}-(k 2^{n-1}+j)}-\binom{r}{\frac{r}{2}-((k+1)2^{n-1}-j)    }  \right)                \cos\left(\frac{j\pi}{2^{n-1}}     \right)\cos\left(\frac{\pi}{2^n}\right)\Biggr)\\\notag
&=\frac{1}{2^{r-1}}\Biggl(\frac{1}{2}\sum_{k=0}^{\floor{r/2^n}}(-1)^k\left(\binom{r}{\frac{r}{2}-k2^{n-1}}-\binom{r}{\frac{r}{2}-(k+1)2^{n-1}}\right)\cos\left(\frac{\pi}{2^n}\right)\\\notag\\\notag
&+  \frac{1}{2}\sum_{j=1}^{2^{n-2}-1}     \sum_{k=0}^{\floor{r/2^n}}\\\notag
&(-1)^k\left(              \binom{r}{\frac{r}{2}-(k 2^{n-1}+j)}-\binom{r}{\frac{r}{2}-((k+1)2^{n-1}-j)    }  \right)  \left(            \cos\left(\frac{(2j-1)\pi}{2^{n-1}}\right)+\cos\left(\frac{(2j+1)\pi}{2^{n-1}} \right)   \right)\Biggr)\\\notag
&=\frac{1}{2^{r}}\Biggl(\sum_{k=0}^{\floor{r/2^n}}(-1)^k\left(\binom{r}{\frac{r}{2}-k2^{n-1}}-\binom{r}{\frac{r}{2}-(k+1)2^{n-1}}\right)\cos\left(\frac{\pi}{2^n}\right)\\\notag\\\notag
&+\sum_{k=0}^{\floor{r/2^n}}(-1)^k\left(\binom{r}{\frac{r}{2}-(k2^{n-1}+1)}-\binom{r}{\frac{r}{2}-((k+1)2^{n-1}-1)}\right)\cos\left(\frac{\pi}{2^n}\right)\\\notag\\\notag
&+  \sum_{j=1}^{2^{n-2}-2}     \sum_{k=0}^{\floor{r/2^n}}(-1)^k\Biggl(            \left(  \binom{r}{\frac{r}{2}-(k 2^{n-1}+j+1)}+\binom{r}{\frac{r}{2}-(k 2^{n-1}+j)}\right)\\\notag
&-\left(\binom{r}{\frac{r}{2}-((k+1)2^{n-1}-j-1)    }+ \binom{r}{\frac{r}{2}-((k+1)2^{n-1}-j)    }\right) \Biggr)             \cos\left(\frac{(2j+1)\pi}{2^{n-1}} \right) \Biggr)\\\notag
&+\sum_{k=0}^{\floor{r/2^n}}(-1)^k\left(\binom{r}{\frac{r}{2}-(k2^{n-1}+2^{n-2}-1)}-\binom{r}{\frac{r}{2}-((k+1)2^{n-1}-2^{n-2}+1)}\right)\cos\left(\frac{(2^{n-1}-1)\pi}{2^n}\right)\\\notag\\\notag
\end{align}
\begin{align}\notag
&=\frac{1}{2^{r}}\Biggl(\sum_{k=0}^{\floor{r/2^n}}(-1)^k\left(\binom{r}{\frac{r}{2}-k2^{n-1}}-\binom{r}{\frac{r}{2}-(k+1)2^{n-1}}\right)\cos\left(\frac{\pi}{2^n}\right)\\\notag\\\notag
&+\sum_{k=0}^{\floor{r/2^n}}(-1)^k\left(\binom{r}{\frac{r}{2}-(k2^{n-1}+1)}-\binom{r}{\frac{r}{2}-((k+1)2^{n-1}-1)}\right)\cos\left(\frac{\pi}{2^n}\right)\\\notag\\\notag
&+  \sum_{j=1}^{2^{n-2}-2}     \sum_{k=0}^{\floor{r/2^n}}(-1)^k            \left(  \binom{r+1}{\frac{r}{2}-(k 2^{n-1}+j)}-\binom{r+1}{\frac{r}{2}-((k+1) 2^{n-1}-j-1)}\right)  \cos\left(\frac{(2j+1)\pi}{2^{n-1}} \right) \Biggr)\\\notag
&+\sum_{k=0}^{\floor{r/2^n}}(-1)^k\left(\binom{r}{\frac{r}{2}-(k2^{n-1}+2^{n-2}-1)}-\binom{r}{\frac{r}{2}-((k+1)2^{n-1}-2^{n-2}+1)}\right)\cos\left(\frac{(2^{n-1}-1)\pi}{2^n}\right)\\\notag\\\notag
&=\frac{1}{2^{r}}\Biggl( \sum_{j=0}^{2^{n-2}-2}     \sum_{k=0}^{\floor{r/2^n}}(-1)^k            \left(  \binom{r+1}{\frac{r}{2}-(k 2^{n-1}+j)}-\binom{r+1}{\frac{r}{2}-((k+1) 2^{n-1}-j-1)}\right)  \cos\left(\frac{(2j+1)\pi}{2^{n-1}} \right) \Biggr)\\\notag
&+\sum_{k=0}^{\floor{r/2^n}}(-1)^k\left(\binom{r}{\frac{r}{2}-(k2^{n-1}+2^{n-2}-1)}-\binom{r}{\frac{r}{2}-((k+1)2^{n-1}-2^{n-2}+1)}\right)\cos\left(\frac{(2^{n-1}-1)\pi}{2^n}\right)\\\notag\\\notag
&=\frac{1}{2^{r}}\Biggl( \sum_{j=1}^{2^{n-2}-1}     \sum_{k=0}^{\floor{r/2^n}}(-1)^k            \left(  \binom{r+1}{\frac{r}{2}-(k 2^{n-1}+j)}-\binom{r+1}{\frac{r}{2}-((k+1) 2^{n-1}-j-1)}\right)  \cos\left(\frac{(2j+1)\pi}{2^{n-1}} \right) \Biggr)\\\notag
&+\sum_{k=0}^{\floor{r/2^n}}(-1)^k\Biggl(  \left(\frac{\frac{r}{2}+k2^{n-1}+2^{n-2}}{r+1}\right)\binom{r+1}{\frac{r}{2}-(k2^{n-1}+2^{n-2}-1)}\\\notag
&-\left(\frac{\frac{r}{2}-(k+1)2^{n-1}+2^{n-2}}{r+1}\right)\binom{r+1}{\frac{r}{2}-((k+1)2^{n-1}-2^{n-2})}   \Biggr)\cos\left(\frac{(2^{n-1}-1)\pi}{2^n}\right)\Biggr)\\\notag
\end{align}
Comparing with the formula for odd  powers, we see that we need 
\begin{align}\notag
&\sum_{k=0}^{\floor{r/2^n}}(-1)^k\Biggl(  \left(\frac{\frac{r}{2}+k2^{n-1}+2^{n-2}}{r+1}\right)\binom{r+1}{\frac{r}{2}-(k2^{n-1}+2^{n-2}-1)}\\\notag
&-\left(\frac{\frac{r}{2}-(k+1)2^{n-1}+2^{n-2}}{r+1}\right)\binom{r+1}{\frac{r}{2}-((k+1)2^{n-1}-2^{n-2})}\Biggr)\\\notag
&=\sum_{k=0}^{\floor{r/2^n}}(-1)^k\Biggl( \binom{r+1}{\frac{r}{2}-(k2^{n-1}+2^{n-2}-1)}-\binom{r+1}{\frac{r}{2}-((k+1)2^{n-1}-2^{n-2})}\\\notag
\end{align}
Or
\begin{align}\notag
&\sum_{k=0}^{\floor{r/2^n}}(-1)^k\Biggl(  \left(1-\frac{\frac{r}{2}+k2^{n-1}+2^{n-2}}{r+1}\right)\binom{r+1}{\frac{r}{2}-(k2^{n-1}+2^{n-2}-1)}\\\notag
&-\left(1-\frac{\frac{r}{2}-(k+1)2^{n-1}+2^{n-2}}{r+1}\right)\binom{r+1}{\frac{r}{2}-((k+1)2^{n-1}-2^{n-2})}\Biggr)=0\\\notag
\end{align}
But
\begin{align}\notag
& \left(1-\frac{\frac{r}{2}+k2^{n-1}+2^{n-2}}{r+1}\right)\binom{r+1}{\frac{r}{2}-(k2^{n-1}+2^{n-2}-1)}\\\notag
&-\left(1-\frac{\frac{r}{2}-(k+1)2^{n-1}+2^{n-2}}{r+1}\right)\binom{r+1}{\frac{r}{2}-((k+1)2^{n-1}-2^{n-2})}\\\notag
&=\binom{r}{\frac{r}{2}-(k2^{n-1}+2^{n-2})}-\binom{r}{\frac{r}{2}-((k+1)2^{n-1}-2^{n-2})}=0\\\notag
\end{align}
since
\begin{align}\notag
&(k+1)2{n-1}-2^{n-2}=k2^{n-1}+2^{n-1}-2^{n-2}=k2^{n-1}+2^{n-2}\\\notag
\end{align}
\end{proof}
\begin{proposition}\label{evenpowers}
For $n\ge2,\ r\ {\rm even},\ r\ge2$
\begin{align}\notag
&\cos^r\left(\frac{(2i-1)\pi}{2^n}\right)\\\notag
&=\frac{1}{2^{r-1}}\Biggl(\frac{1}{2}\sum_{k=0}^{\floor{r/2^n}}(-1)^k\left(\binom{r}{\frac{r}{2}-k2^{n-1}}-\binom{r}{\frac{r}{2}-(k+1)2^{n-1}}\right)\\\notag
&+\sum_{m=2}^{n-1}\sum_{j=1}^{2^{n-m-1}}  (-1)^{\floor{(2^{n-m-1}+2ij-i-j)/2^{n-m}}}     \sum_{k=0}^{\floor{(r+1)/(2^n)}}\\\notag
&(-1)^k\left(              \binom{r}{\frac{r}{2}-(k 2^{n-1}+2^{m-2}(2j-1))}-\binom{r}{\frac{r}{2}-((k+1)2^{n-1}-2^{m-2}(2j-1))    }  \right)       \\\notag
&\cos\Biggl(      \frac{   \Biggl(  2\Biggl[ \Biggl(  (-1)^{\floor{ (2ij-i-j)/2^{n-m-1} } } ( (2 i j-i-j)\pmod*{2^{n-m-1}}+1)\Biggr)\pmod*{2^{n-m-1}+1}  \Biggr ]  -1   \Biggr)\pi   }    {2^{n-m+1}}     \Biggr )\Biggr)\\\notag
\end{align}
\end{proposition}
\begin{proposition}
Assume $n\ge2,\ r\ {\rm even},\ r\ge2$. Regarding the $2^{n-2}x2^{n-2}$ matrices, $M_n$ which satisfy $$\cos^r(\frac{(2i-1)\pi}{2^n})=\frac{1}{2^{r-1}} \sum_{k=1}^{2^{n-2}}M_n(i,k)\cos(\frac{(k-1)\pi}{2^n})$$
For $1\le i\le 2^{n-2}$
\begin{align}\notag
M_n[i,1]&=\frac{1}{2}\sum_{k=0}^{\floor{r/2^n}}(-1)^k\left(\binom{r}{\frac{r}{2}-k2^{n-1}}-\binom{r}{\frac{r}{2}-(k+1)2^{n-1}}\right)\\\notag
\end{align}
Let
\begin{align}\notag
p &=2ij-i-j+1\\\notag
s &= \floor{\frac{(p - 1)}{2^{n-m-1}}}\\\notag
d&=(p-1)\pmod*{2^{n-m-1}}\\\notag
h &=1+2^{m-2}\Biggl( ((-1)^s (d +1))\pmod*{2^{n-m-1}}\Biggr)\\\notag
q&= \floor{\frac{(p - 1+2^{n-m-1})}{2^{n-m}}}\\\notag
\end{align}
where $m=2,...,n-1$, $j=1,...,2^{n-m-1}$, $i=1,...,2^{n-2}$,
then
\begin{align}\notag
M_n[i,h]&= (-1)^{q}   \sum_{k=0}^{\floor{  r/2^n  } }(-1)^k\left(              \binom{r}{\frac{r}{2}-(k 2^{n-1}+2^{m-2}(2j-1))}-\binom{r}{\frac{r}{2}-((k+1)2^{n-1}-2^{m-2}(2j-1))    }  \right) \\\notag
\end{align}
\end{proposition}
\begin{corollary}
Regarding the $2^{n-2}x2^{n-2}$ matrices, $M_n$ which satisfy $$\cos^r(\frac{(2i-1)\pi}{2^n})=\frac{1}{2^{r-1}} \sum_{j=1}^{2^{n-2}}M_n(i,j)\cos(\frac{(j-1)\pi}{2^{n-1}})$$ where $r$ even, $r\ge 2:$
For $j\ge 2$ let
\begin{align}\notag
m&=\gcd(j-1,2^{n-1})\\\notag
p&=\frac{j-1+m}{2m}\\\notag
q&=(i+p-1)(2i-1)^{   \frac{2^{n+1}}{m}-1  }      \pmod*{\frac{2^{n-2}}{m}}\\\notag
s&=\floor{\frac{2^{n+1}+m(2iq-i-q)}{2^{n-2}}}\\\notag
\end{align}
then, for $1\le i\le 2^{n-2}$ and $2\le j \le 2^{n-2}$
\begin{align}\notag
Mn[i,1]&=\frac{1}{2}\sum_{k=0}^{\floor{r/2^n}}(-1)^k\left(\binom{r}{\frac{r}{2}-k2^{n-1}}-\binom{r}{\frac{r}{2}-(k+1)2^{n-1}}\right)\\\notag
M_n[i,j]&= (-1)^{s}\left(       \sum_{k=0}^{\floor{  r/(2^n)  } }(-1)^k\left(              \binom{r}{\frac{r}{2}-(k 2^{n-1}+m(2q-1))}-\binom{r}{\frac{r}{2}-((k+1)2^{n-1}-m(2q-1))   }  \right)     \right)  \\\notag
\end{align}
\end{corollary}

Regarding even powers, M. Merca gave the following result.
\begin{theorem}
Let $n$ and $p$ be two positive integers. Then
$$\sum_{k=1}^{\floor{\frac{n-1}{2}}}\cos^{2p}\left(\frac{k\pi}{n}\right)=-\frac{1}{2}+\frac{n}{2^{2p+1}}\sum_{k=-\floor{\frac{p}{n}}}^{\floor{\frac{p}{n}}}\binom{2p}{p+kn}$$
\end{theorem}
\begin{proof}
See ref. 19.
\end{proof}
\section{Negative Odd Powers}
The case $r=-1$ has matrix M with the magnitude of all entries equal to $1$, e.g.,

\begin{equation}\notag
\left(
\begin{array}{cc}
\cos^{-1}\left(\frac{\pi}{16}\right)   \\
\cos^{-1}\left(\frac{3\pi}{16}\right)  \\
\cos^{-1}\left(\frac{5\pi}{16}\right)  \\
\cos^{-1}\left(\frac{7\pi}{16}\right)  \\
\end{array}
\right)=2
\left(
\begin{array}{cccc}
 1 & -1 & 1 & -1 \\
 -1 & 1 & 1 & 1 \\
 1 & -1 & 1 & 1 \\
1 & 1 & 1 & 1 \\
\end{array}
\right)
\left(
\begin{array}{cc}
\cos\left(\frac{\pi}{16}\right)   \\
\cos\left(\frac{3\pi}{16}\right)  \\
\cos\left(\frac{5\pi}{16}\right)  \\
\cos\left(\frac{7\pi}{16}\right)  \\
\end{array}
\right)
\end{equation}
We remind the viewer that the above equation is equivalent to
\begin{equation}\notag
\left(
\begin{array}{cc}
\sin^{-1}\left(\frac{\pi}{16}\right)   \\
\sin^{-1}\left(\frac{3\pi}{16}\right)  \\
\sin^{-1}\left(\frac{5\pi}{16}\right)  \\
\sin^{-1}\left(\frac{7\pi}{16}\right)  \\
\end{array}
\right)=2
\left(
\begin{array}{cccc}
 1 & 1 & 1 & 1 \\
 1 & 1 & -1 & 1 \\
 1 & 1 & 1 & -1 \\
-1 & 1 & -1 & 1 \\
\end{array}
\right)
\left(
\begin{array}{cc}
\sin\left(\frac{\pi}{16}\right)   \\
\sin\left(\frac{3\pi}{16}\right)  \\
\sin\left(\frac{5\pi}{16}\right)  \\
\sin\left(\frac{7\pi}{16}\right)  \\
\end{array}
\right)
\end{equation}
\begin{proposition}
Assume $n\ge2,\ r\ {\rm odd},\ r\ge1$. Regarding the $2^{n-2}x2^{n-2}$ matrices, $M_n$ which satisfy $$\cos^{-1}(\frac{(2j-1)\pi}{2^n})=2 \sum_{k=1}^{2^{n-2}} M_n(j,k)\cos(\frac{(2k-1)\pi}{2^n})$$

then
\begin{align}\notag
M_n[i,j]&= (-1)^{    1+\floor{      \frac{   ((1-i-j)(1+2^{n-1}-2i)^{2^{n-2}-1}  } {2^{n-1}}    }                   }\\\notag
\end{align}
\end{proposition}
The cases where $r=-3,\ -5$ were worked out previously (see ref. 20):

\begin{proposition}
Regarding the $2^{n-2}x2^{n-2}$ matrices, $M_n$ which satisfy $$\frac{1}{(\sin(\frac{(2j-1)\pi}{2^n}))^3}=8 \sum_{k=1}^{2^{n-2}}M_n(j,k)\sin(\frac{(2k-1)\pi}{2^n})$$
Let
\begin{align}\notag
p &=2ij-i-j+1\\\notag
k &= \floor     {\frac{(p - 1)}  {2^{n-2}}     }\\\notag
m &= (-1)^k (p - k(2^{n-2}))\pmod*{(2^{n-2}+1)}\notag
\end{align}
then
\begin{align}\notag
M_n[i,m]&= (1/2) (-1)^{ k }(2^{n-1}j - j^2 + j -2^{n-2})
\end{align}
\end{proposition}
\begin{corollary}
\begin{align}\notag
M_n[i,j]&= (-1)^{  \floor{   \frac{((i+j-1)(2i-1)^{(2^{n-2}-1)} }{2^{n-1}}   }     }\Biggl[ \Bigl[  (2^{n-1}+1)\left((i+j-1)(2i-1)^{2^{n-2}-1}\pmod*{2^{n-1}}\right)\\\notag
&-    \left( (i+j-1)(2i-1)^{2^{n-2}-1}\pmod*{2^{n-1}}\right)^{2}-2^{n-2}  \Bigr]\Biggr]\\\notag
\end{align}
\end{corollary}
\begin{corollary}
Let $S_n=\sum_{i=1}^{2^{n-2}}\sin^{-3}\left(\frac{(2i-1)\pi}{2^n}\right)$
\begin{align}\notag
S_n=&2^2\sum_{i=1}^{2^{n-2}}\sum_{j=1}^{2^{n-2}}(-1)^{\floor*{\frac{2 i j-i-j}{2^{n-1}}}}(2^{n-1}j-j^2+j-2^{n-2})\sin(      \frac{ (  2[ (2 i j-i-j)\pmod*{2^{n-1}}] +1 )\pi } {2^{n}}      )\\\notag
&=2^2\sum_{j=1}^{2^{n-2}}(2^{n-1}j-j^2+j-2^{n-2})\sum_{i=1}^{2^{n-2}}(-1)^{\floor*{\frac{2 i j-i-j}{2^{n-1}}}}\sin(      \frac{ (  2[ (2 i j-i-j)\pmod*{2^{n-1}}] +1 )\pi } {2^{n}}      )\\\notag
\end{align}
\end{corollary}
\begin{proof}
Sum terms in previous Proposition.
\end{proof}
\begin{proposition}\label{commute}
Let $$S_n=\sum_{k=1}^{2^{n-2}}\frac{1}{\sin^r(\frac{(2k-1)\pi}{2^n})}$$
Regarding the $2^{n-2}x2^{n-2}$ matrices, $M_n$ which satisfy $$\frac{1}{(\sin(\frac{(2j-1)\pi}{2^n}))^r}=2^r \sum_{k=1}^{2^{n-2}}M_n(j,k)\sin(\frac{(2k-1)\pi}{2^n})$$
$$S_n=2^{r-1} \sum_{j=1}^{2^{n-2}}M(1,j)\frac{1}{\sin(\frac{(2j-1)\pi}{2^n})}$$
\end{proposition}
\begin{proof}
We will show how the proof goes with an example having $n=4$, $r=3$.
\begin{equation}\notag
S_4=\left(
\begin{array}{cccc}
 1 & 1 & 1 & 1 \\
\end{array}
\right)
\left(
\begin{array}{c}
\sin^{-3}\left(\frac{\pi}{16}\right)   \\
\sin^{-3}\left(\frac{3\pi}{16}\right)  \\
\sin^{-3}\left(\frac{5\pi}{16}\right)  \\
\sin^{-3}\left(\frac{7\pi}{16}\right)  \\
\end{array}
\right)
\end{equation}
\begin{equation}\notag
\left(
\begin{array}{c}
\sin^{-3}\left(\frac{\pi}{16}\right)   \\
\sin^{-3}\left(\frac{3\pi}{16}\right)  \\
\sin^{-3}\left(\frac{5\pi}{16}\right)  \\
\sin^{-3}\left(\frac{7\pi}{16}\right)  \\
\end{array}
\right)
=8\left(
\begin{array}{cccc}
 2 & 5 & 7 & 8 \\
 7 & 2 & -8 & 5 \\
 5 & 8 & 2 & -7 \\
 -8 & 7 & -5 & 2 \\
\end{array}
\right)
\left(
\begin{array}{c}
\sin\left(\frac{\pi}{16}\right)   \\
\sin\left(\frac{3\pi}{16}\right)  \\
\sin\left(\frac{5\pi}{16}\right)  \\
\sin\left(\frac{7\pi}{16}\right)  \\
\end{array}
\right)
\end{equation}
Multiplying both sides of the above equation on the left by
\begin{equation}\notag
\left(
\begin{array}{cccc}
 1 & 1 & 1 & 1 \\
 1 & 1 & -1 & 1 \\
 1 & 1 & 1 & -1 \\
-1 & 1 & -1 & 1 \\
\end{array}
\right)
\end{equation}
We see that $S_n$ occupies the first entry in the resulting vector. However, we know from Proposition \ref {commute} that these matricies commute. 
\begin{equation}\notag
\left(
\begin{array}{cccc}
 1 & 1 & 1 & 1 \\
 1 & 1 & -1 & 1 \\
 1 & 1 & 1 & -1 \\
-1 & 1 & -1 & 1 \\
\end{array}
\right)
\left(
\begin{array}{c}
\sin^{-3}\left(\frac{\pi}{16}\right)   \\
\sin^{-3}\left(\frac{3\pi}{16}\right)  \\
\sin^{-3}\left(\frac{5\pi}{16}\right)  \\
\sin^{-3}\left(\frac{7\pi}{16}\right)  \\
\end{array}
\right)
=8\left(
\begin{array}{cccc}
 2 & 5 & 7 & 8 \\
 7 & 2 & -8 & 5 \\
 5 & 8 & 2 & -7 \\
 -8 & 7 & -5 & 2 \\
\end{array}
\right)
\left(
\begin{array}{cccc}
 1 & 1 & 1 & 1 \\
 1 & 1 & -1 & 1 \\
 1 & 1 & 1 & -1 \\
-1 & 1 & -1 & 1 \\
\end{array}
\right)
\left(
\begin{array}{c}
\sin\left(\frac{\pi}{16}\right)   \\
\sin\left(\frac{3\pi}{16}\right)  \\
\sin\left(\frac{5\pi}{16}\right)  \\
\sin\left(\frac{7\pi}{16}\right)  \\
\end{array}
\right)
\end{equation}
\begin{equation}\notag
=4\left(
\begin{array}{cccc}
 2 & 5 & 7 & 8 \\
 7 & 2 & -8 & 5 \\
 5 & 8 & 2 & -7 \\
 -8 & 7 & -5 & 2 \\
\end{array}
\right)
\left(
\begin{array}{c}
\sin^{-3}\left(\frac{\pi}{16}\right)   \\
\sin^{-3}\left(\frac{3\pi}{16}\right)  \\
\sin^{-3}\left(\frac{5\pi}{16}\right)  \\
\sin^{-3}\left(\frac{7\pi}{16}\right)  \\
\end{array}
\right)
\end{equation}\notag
The general case uses the same argument.

\end{proof}
\begin{corollary}
Let
\begin{align}\notag
S(s,n)&=\sum_{j=1}^{2^{n-2}}\frac{1}{\sin^s\left(\frac{(2j-1)\pi}{2^n}\right)}\\\notag
\end{align}
Then
\begin{align}\notag
S(2,n)&=\frac{1}{2}2^{2n-2}\\\notag
S(3,n)&=\sum_{j=1}^{2^{n-2}}\left(-2j^2+2(2^{n-1}+1)j-2^{n-1}\right)\frac{1}{\sin[\frac{(2j-1)\pi}{2^n}]}\\\notag
S(4,n)&=\frac{1}{6}\left(2^{4n-4}+2(2^{2n-2})\right)\\\notag
S(5,n)&=\frac{1}{3}\sum_{j=1}^{2^{n-2}}\biggl(2 j^4-4(2^{n-1}+1)j^3+2(3( 2^{n-1})-1)j^2+2(2^{3n-3}+2^{n-1}+2)j\\\notag
&-(2^{3n-3}+2(2^{n-1}))\biggr)\frac{1}{\sin[\frac{(2j-1)\pi}{2^n}]}\\\notag
S(6,n)&=\frac{1}{30}(2(2^{6n-6})+5(2^{4n-4})+8(2^{2n-2}))\\\notag
S(7,n)&=\frac{1}{45}\sum_{j=1}^{2^{n-2}}\biggl(  -4 j^6+12(2^{n-1}+1)j^5-10(3(2^{n-1})-2)j^4-20(2^{3n-3}+2(2^{n-1})+3)j^3 \\\notag
&+2(15(2^{3n-3})+45(2^{n-1})-8)j^2 +4( 3(2^{5n-5})+5(2^{3n-3})+4(2^{2n-1})+12 )j\\\notag
&-3( 2(2^{5n-5})+5(2^{3n-3})+8(2^{n-1}) )\biggr)\frac{1}{\sin[\frac{(2j-1)\pi}{2^n}]}\\\notag
S(8,n)&=\frac{1}{630}\left(17(2^{8n-8})+56(2^{6n-6})+98(2^{4n-4})+144(2^{2n-2})\right)\\\notag
\end{align}
\end{corollary}
\begin{proof}
See ref. 20.
\end{proof}

The Mathematica code for producing the matrix for terms of $M_m$, $S_m$ is given in the appendix. 

 \begin{proposition}
Regarding the $2^{n-2}x2^{n-2}$ matrices, $M_n$ which satisfy $$\frac{1}{(\sin(\frac{(2j-1)\pi}{2^n}))^5}=32 \sum_{k=1}^{2^{n-2}}M_n(j,k)\sin(\frac{(2k-1)\pi}{2^n})$$
Let
\begin{align}\notag
p &= i + (j - 1) (2 (i - 1) + 1)\\\notag
k &= \floor{\frac{(p - 1)}{2^{n-2}}}\\\notag
m &= (-1)^k (p - k(2^{n-2}))\pmod*{(2^{n-2}+1)}\notag
\end{align}
then
\begin{align}\notag
M_n[i,m]&=  (-1)^{ k }\Bigl(\frac{1}{24} (j-1) \left(j-2^{n-1}\right) \left(j^2-j \left(2^{n-1}+1\right)+2 \left(-6\ 2^{2 n-5}+2^{2 n-4}-1\right)\right)\\\notag
&+\frac{1}{3} 2^{n+1} \left(4^{n-4}-1\right)+11\ 2^{n-4}\Bigr)\\\notag
&= (-1)^{ k }\frac{1}{24} \left(j^4-2 j^3 \left(2^{n-1}+1\right)+j^2 \left(3\ 2^{n-1}-1\right)+2 j \left(2^{n-2}+2^{3 n-4}+1\right)-2^{n-1} \left(2^{2 n-3}+1\right)\right)\\\notag
\end{align}
\end{proposition} 
The Mathematica code for producing the matrix for terms of $M_{m+2}$, $(2^{-5})S_m$ is given in the Appendix

 \section{Some Trig Results}
 \begin{lemma}
$$x-\frac{x^3}{6}<\sin(x)<x-\frac{2x^3}{3\pi^2}\ \ \rm{;}\ x\in(0,\pi/2)$$\notag
\end{lemma}
\begin{proof}
Reference R. Kl\'en et al [17]
\end{proof}
\begin{lemma}
Assume $n\in\mathbb{Z}$, $n>0$ and $a,b\in\mathbb{R}$ then if $\sin(\frac{d}{2})\ne 0$
$$\sum_{i=0}^{n-1}\sin(a+id)=\frac{\sin(\frac{nd}{2})\sin(a+\frac{(n-1)d}{2})}{\sin(\frac{d}{2})}$$
\end{lemma}
\begin{proof}
Ref [18].
\end{proof}

\begin{lemma}
For $n\in \mathbb{Z}^+$
\begin{align}\notag
\sin(n \theta)&=\sum_{r=0}^{\floor{(n-1)/2}}(-1)^r\binom{n}{2r+1}\cos^{n-2r-1}(\theta)\sin^{2r+1}(\theta)\\\notag
\cos(n \theta)&=\sum_{r=0}^{\floor{(n)/2}}(-1)^r\binom{n}{2r}\cos^{n-2r}(\theta)\sin^{2r}(\theta)\\\notag
\end{align}
\end{lemma}
\begin{proof}
Note that it is equivalent to say
\begin{align}\notag
\sin(n \theta)&=\sum_{r=0}^{n}(-1)^r\binom{n}{2r+1}\cos^{n-2r-1}(\theta)\sin^{2r+1}(\theta)\\\notag
\cos(n \theta)&=\sum_{r=0}^{n}(-1)^r\binom{n}{2r}\cos^{n-2r}(\theta)\sin^{2r}(\theta)\\\notag
\end{align}
Since, in the equation for $\sin(n \theta)$,  $r>\floor{(n-1)/2}\implies \binom{n}{2r+1}=0$ since then $2r+1>n$ and $r\in \mathbb{Z}$, similarly for $\cos(n \theta)$. The proof then procedes by induction. True for $n=1$. Assume true for $n$ and show for $n+1$.
\begin{align}\notag
\sin((n+1) \theta)&=\sin(n\theta+\theta)
&=\sin(n\theta)\cos(\theta)+\cos(n\theta)\sin(\theta)\\\notag
&=\sum_{r=0}^{n}(-1)^r\binom{n}{2r+1}\cos^{n-2r-1}(\theta)\sin^{2r+1}(\theta)\cos(\theta)+\sum_{r=0}^{n}(-1)^r\binom{n}{2r}\cos^{n-2r}(\theta)\sin^{2r}(\theta)\sin(\theta)\\\notag
&=\sum_{r=0}^{n}(-1)^r\binom{n}{2r+1}\cos^{n-2r}(\theta)\sin^{2r+1}(\theta)+\sum_{r=0}^{n}(-1)^r\binom{n}{2r}\cos^{n-2r}(\theta)\sin^{2r+1}(\theta)\\\notag
&=\sum_{r=0}^{n+1}(-1)^r\binom{n+1}{2r}\cos^{n-2r}(\theta)\sin^{2r+1}(\theta)\\\notag
\end{align}
by Pascal's identity.
\end{proof}
Just as Newton's generalized binomial theorem follows from the binomial theorem, we have
\begin{corollary}
For $r\in \mathbb{C}$ with $|\cos(\theta)|>|\sin(\theta)|$
\begin{align}\notag
\sin(r \theta)&=\sum_{j=0}^{\infty}(-1)^j\binom{r}{2j+1}\cos^{r-2j-1}(\theta)\sin^{2j+1}(\theta)\\\notag
\cos(r \theta)&=\sum_{j=0}^{\infty}(-1)^j\binom{r}{2j}\cos^{r-2j}(\theta)\sin^{2j}(\theta)\\\notag
\end{align}
\begin{proof}
Note that if $f(\theta)=\sin(r \theta)$ and $g(\theta)=\cos(r \theta)$ then
\begin{align}\notag
f'&=r g\\\notag
g'&=-r f\\\notag
\end{align}
$f$ and $g$ both satisfy
$$f''=-r^2 f$$
The general solution to this differential equation is 
$$c_1 e^{i r \theta}+c_2 e^{-ir\theta}$$
with constants resolved by evaluation at specific points.
One shows that the series expansions also satisy this equation.
\begin{align}\notag
&\frac{d}{d\theta}\sum_{j=0}^{\infty}(-1)^j\binom{r}{2j+1}\cos^{r-2j-1}(\theta)\sin^{2j+1}(\theta)\\\notag
&\ =\sum_{j=0}^{\infty}(-1)^j(\binom{r}{2j+1}(-(r-2j-1)\cos^{r-2j-2}(\theta)\sin^{2j+2}(\theta)+(2j+1)\cos^{r-2j}(\theta)\sin^{2j}(\theta))\\\notag
&\ =\sum_{j=0}^{\infty}(-1)^j( r\binom{r-1}{2j+1}\cos^{r-2j-2}(\theta)\sin^{2j+2}(\theta)+r\binom{r-1}{2j}\cos^{r-2j}(\theta)\sin^{2j}(\theta))\\\notag
&\ =r\sum_{j=0}^{\infty}(-1)^j(\binom{r-1}{2j-1}\cos^{r-2j}(\theta)\sin^{2j}(\theta)+\binom{r-1}{2j}\cos^{r-2j}(\theta)\sin^{2j}(\theta))\\\notag
&\ =r\sum_{j=0}^{\infty}(-1)^j\left(\binom{r-1}{2j-1}+\binom{r-1}{2j}\right)\cos^{r-2j}(\theta)\sin^{2j}(\theta)\\\notag
&\ =r\sum_{j=0}^{\infty}(-1)^j\binom{r}{2j}\cos^{r-2j}(\theta)\sin^{2j}(\theta)\\\notag
\end{align}
The fourth line in the proof results from re-indexing $j-1\to j$ and the fact that $\binom{r-1}{-1}=0$. The last line results from Pascal's identity. The proof for the cosine expansion is similar. We require  $|\cos(\theta)|>|\sin(\theta)|$ to ensure convergence.
\end{proof}
\end{corollary}
\begin{corollary}
Assume  $|\cos(\theta)|>|\sin(\theta)|$, then
\begin{align}\notag
\frac{1}{\cos^r(\theta)}&=\frac{1}{\sin(r \theta)}\sum_{j=0}^{\infty}(-1)^j\binom{r}{2j+1}\cos^{-2j-1}(\theta)\sin^{2j+1}(\theta)\\\notag
&=\frac{1}{\cos(r \theta)}\sum_{j=0}^{\infty}(-1)^j\binom{r}{2j}\cos^{-2j}(\theta)\sin^{2j}(\theta)\\\notag
\end{align}
\end{corollary}
\begin{corollary}
Assume  $|\sin(\theta)|>|\cos(\theta)|$, then
\begin{align}\notag
\frac{1}{\sin^r(\theta)}&=\frac{1}{\sin(r(\frac{\pi}{2}- \theta))}\sum_{j=0}^{\infty}(-1)^j\binom{r}{2j+1}\sin^{-2j-1}(\theta)\cos^{2j+1}(\theta)\\\notag
&=\frac{1}{\cos(r(\frac{\pi}{2}- \theta))}\sum_{j=0}^{\infty}(-1)^j\binom{r}{2j}\sin^{-2j}(\theta)\cos^{2j}(\theta)\\\notag
\end{align}

\end{corollary}

\begin{proposition}\label{sineexpansion}
For  $r\in \mathbb{C}$, $0\le\theta\le\frac{\pi}{2}$
\begin{align}\notag
\frac{1}{\sin^r(\theta)}&=2^{r/2}\sum_{j=0}^\infty (-1)^j \binom{-r/2}{j}\cos^j(2\theta)\\\notag
\end{align}
\end{proposition}
\begin{proof}
\begin{align}\notag
\sin(\theta)&=\sqrt{\frac{1-\cos(2\theta)}{2}}\\\notag
\implies \\\notag
\frac{1}{\sin^r}(\theta)&=\left(\frac{1-\cos(2\theta)}{2}\right)^{-r/2}\\\notag
&=2^{r/2}\left(1-\cos(2\theta)\right)^{-r/2}\\\notag
\end{align}
Hence we can apply the generalized binomial theorem to the expression in the preceeding line.
\end{proof}

\section{An Interesting Formula For The Riemann Zeta Function}
\begin{proposition}\label{basic}
Assume $s\in\mathbb{C}$, $\rm{Re}(s)>1$ then
$$\zeta(s)=\lim_{n\to\infty}\left(\frac{2^s\pi^s}{2^{s}-1}\right) \sum _{i=1}^{2^{n-2}}\frac{1}{ \left(2^n\sin\left(\frac{\left(2 i-1\right)\pi}{2^n}\right)\right)^s} $$\notag\\
\end{proposition}
\begin{proof}
See ref. 20.
\end{proof}

\begin{proposition}
For $n,p\in \mathbb{Z}$, $n\ge 2$, $p>0$
\begin{align}\notag
\sum_{i=1}^{2^{n-2}}\Biggl(\Biggl(2\cos\left(\frac{(2i-1)\pi}{2^n}\right)\Biggr)^{2p}\Biggr)/2^{n-2}&=\sum_{k=0}^{\floor{p/2^{n-1}}}(-1)^k\left(\binom{2p}{p-k2^{n-1}}-\binom{2p}{p-(k+1)2^{n-1}}\right)\\\notag
\end{align}
\end{proposition}
\begin{proof}
The field extension over $\mathbb{C}$ that is the splitting field for $\cos\left(\frac{(2i-1)\pi}{2^n}\right)$ has basis $\cos\left(\frac{j \pi}{2^n}\right)$, $ 0\le j\le 2^{n-1}-1$. The automorphism that sends $\cos\left(\frac{j \pi}{2^n}\right)$ to $\cos\left(\frac{(2i-1)\pi}{2^n}\right)$ is given by $\cos\left(\frac{(2i-1)\pi}{2^n}\right)=(-1)^i p_i\left(\cos\left(\frac{j \pi}{2^n}\right)\right))$. Hence if, as in proposition \ref{evenpowers}, $$\cos^r\left(\frac{j \pi}{2^n}\right)=\sum_{j=0}^{2^{n-2}-1}c_j \cos\left(\frac{j \pi}{2^{n-1}}\right)$$ then 
\begin{align}\notag
\cos^r\left(\frac{(2i-1) \pi}{2^n}\right)&=\left((-1)^i p_i\left(\cos\left(\frac{j \pi}{2^n}\right)\right)\right)^r\\\notag
&=\sum_{j=0}^{2^{n-2}-1}c_j(-1)^i p_i\left( \cos\left(\frac{j \pi}{2^{n-1}}\right)\right)\\\notag
&=\sum_{j=0}^{2^{n-2}-1}c_j\left( \cos\left(\frac{(2i-1)j \pi}{2^{n-1}}\right)\right)\\\notag
\end{align}
Assume $j$ is odd. The odd numbers form a group modulo $2^{n-1}$. Hence the set of numbers $\{(2i-1)j \pmod*{2^{n-1}},1\le i\le 2^{n-2}\}$ equals $\{(2i-1),1\le i\le 2^{n-2}\}$. Hence for each $(2i-1)j$ there is a $(2k-1)j$ such that $$(2k-1)j=2^{n-1}-(2i-1)j \pmod*{2^{n-1}}$$ This means that when we sum over $i$, for every $\cos\left(\frac{(2i-1)j\pi}{2^n}\right)$ there is a $\cos\left(\frac{(2k-1)j\pi}{2^n}\right)=-\cos\left(\frac{(2i-1)j\pi}{2^n}\right)$. This implies the cosine terms cancel out. 
Now assume $j$ is even, $j=2^m h$ for some $1\le m<n-2$ and $h$ odd, $h>0$.
\begin{align}\notag
\cos\left(\frac{(2i-1)j \pi}{2^{n-1}}\right)&=\cos\left(\frac{(2i-1)2^m h \pi}{2^{n-1}}\right)\\\notag
&=\cos\left(\frac{(2i-1)h \pi}{2^{n-m-1}}\right)\\\notag
\end{align}
As before, we see that for each $(2i-1)h$ there is a $(2k-1)h$ such that $(2k-1)h=2^{n-m-1}-(2i-1)j\pmod*{2^{n-m-1}}$ and in the sum the cosine terms cancel out.
\end{proof}
Note the equality implies the left hand side is an integer.
\begin{proposition}
\begin{align}\notag
\zeta(s)&=\lim_{n\to\infty}\lim_{m\to\infty}\frac{2^{3s/2-n s+n-3}\pi^s}{2^s-1}\sum_{p=0}^m \frac{1}{2^{2p-1}}\binom{-\frac{s}{2}}{2p}\sum_{k=0}^{\floor{(p)/2^{n-2}}}(-1)^k \left(\binom{2p}{p-k2^{n-2}}-\binom{2p}{p-(k+1)2^{n-2}}\right)\\\notag
\end{align}
\end{proposition}
\begin{proof}
In Proposition \ref{basic} we can use Proposition \ref{sineexpansion} to represent each power of sine as a series in powers of cosine. We can then use Propositions \ref{evenpowers} and \ref{oddpowers} to convert powers of cosine into values of cosine. In Proposition \ref{basic} we are summing $$\sum _{i=1}^{2^{n-2}}\frac{1}{ \left(2^n\sin\left(\frac{\left(2 i-1\right)\pi}{2^n}\right)\right)^s}$$ Note that in Proposition \ref{sineexpansion}, the power of $\sin(\theta)$ gets represented in terms of powers of $\cos(2\theta)$. We then convert those powers of cosine into sums of cosine.  The values of the arguments $\theta$ to the sine function in Proposition 1.1 satisfy $0<\theta<\pi/2$. Our final sum will involve terms $\cos(2\theta)$, with arguments $0<2\theta<\pi$. The result is that the set of values $\cos(2\theta)$, $2\theta>\pi/2$ are the negative of the values $\cos(2\theta)$, $2\theta<\pi/2$. Thus, the sum of these values raised to an odd power equals zero. The case of an even power is resolved in the previous proposition.
\begin{align}\notag
\zeta(s)&=\lim_{n\to\infty}\left(\frac{2^{s(1-n)}\pi^s}{2^{s}-1}\right) \sum _{i=1}^{2^{n-2}}\frac{1}{ \left(\sin\left(\frac{\left(2 i-1\right)\pi}{2^n}\right)\right)^s} \\\notag
&=\lim_{n\to\infty}\lim_{q\to\infty}\left(\frac{2^{s(1-n)}\pi^s}{2^{s}-1}\right) 2^{s/2}\sum _{i=1}^{2^{n-2}}\sum_{p=0}^q (-1)^p \binom{-s/2}{p}\cos^p\left(\frac{\left(2 i-1\right)\pi}{2^{n-1}}\right) \\\notag
&=\lim_{n\to\infty}\lim_{q\to\infty}\left(\frac{2^{s(1-n)}\pi^s}{2^{s}-1}\right) 2^{s/2}\sum_{p=0}^q \Biggl(- \binom{-s/2}{2p+1}\sum _{i=1}^{2^{n-2}}\cos^{2p+1}\left(\frac{\left(2 i-1\right)\pi}{2^{n-1}}\right)\\\notag
&+\binom{-s/2}{2p}\sum _{i=1}^{2^{n-2}}\cos^{2p}\left(\frac{\left(2 i-1\right)\pi}{2^{n-1}}\right)\Biggr) \\\notag
&=\lim_{n\to\infty}\lim_{q\to\infty}\left(\frac{2^{3s/2-ns+n-3}\pi^s}{2^{s}-1}\right)\sum_{p=0}^q
\binom{-s/2}{2p}\frac{1}{2^{2p-1}}\sum_{k=0}^{\floor{p/2^{n-2}}}(-1)^k\left(\binom{2p}{p-k2^{n-2}}-\binom{2p}{p-(k+1)2^{n-2}}\right)\\\notag
\end{align}
\end{proof}
\begin{corollary}
\begin{align}\notag
\zeta(s)&=\lim_{n\to\infty}\lim_{m\to\infty}\frac{2^{3s/2-n s+n-3}\pi^s}{2^s-1}\sum_{p=0}^m \frac{1}{2^{2p-1}}\binom{-\frac{s}{2}}{2p}\left(\binom{2p}{p}+2\sum_{k=1}^{\floor{p/2^{n-2}}}(-1)^k\binom{2p}{p-k 2^{n-2}}\right)\\\notag
&=\lim_{n\to\infty}\lim_{m\to\infty}\frac{2^{3s/2-n s+n-3}\pi^s}{2^s-1}\sum_{p=0}^m \frac{1}{2^{2p-1}}\left(\frac{s}{2}\right)_{2p}\left(\frac{1}{(p!)^2}+2\sum_{k=1}^{\floor{p/2^{n-2}}}(-1)^k\frac{1}{(p-k2^{n-2})!(p+k2^{n-2})!}\right)\\\notag
&=\Gamma\left(1-\frac{s}{2}\right)\lim_{n\to\infty}\lim_{m\to\infty}\frac{2^{3s/2-n s+n-3}\pi^s}{2^s-1}\sum_{p=0}^m \frac{1}{2^{2p-1}}\frac{1}{\Gamma(1-s/2-2p)}\biggl(\frac{1}{(p!)^2}\\\notag
&+2\sum_{k=1}^{\floor{p/2^{n-2}}}(-1)^k\frac{1}{(p-k2^{n-2})!(p+k2^{n-2})!}\biggr)\\\notag
\end{align}
where $\left(\frac{s}{2}\right)_{2p}$ refers to the Pochhammer function or ascending factorial.
\end{corollary}
\begin{proof}
Follows from previous proposition.
\end{proof}

\begin{corollary}
Assume $j\in\mathbb{Z}^+$.
\begin{align}\notag
\lim_{n\to\infty}\lim_{m\to\infty}&\sum_{p=0}^m \frac{1}{2^{2p-1+n(2j-1)}}\left(2p+j-1\right)!\left(\frac{1}{(p!)^2}+2\sum_{k=1}^{\floor{p/2^{n-2}}}(-1)^k\frac{1}{(p-k2^{n-2})!(p+k2^{n-2})!}\right)\\\notag
&=(-1)^{j+1}\frac{2^{2j}-1}{2^{j-2}}\frac{(j-1)!}{(2j)!}\rm{B}_{2j}\\\notag
\end{align}
where $\rm{B}_{2j}$ is the 2j-th Bernoulli number.
\end{corollary}
\begin{proof}
Follows applying previous proposition to $\zeta(2j)$, where $j\in\mathbb{Z}^+$.
\end{proof}
\begin{proposition}
\begin{align}
\zeta(3)&=\lim_{n\to\infty}\frac{\pi^3}{7*2^{3n-4}}\sum_{j=1}^{2^{n-2}}\left(-j^2+(2^{n-1}+1)j-2^{n-2}\right)\frac{1}{\sin[\frac{(2j-1)\pi}{2^n}]}\\\notag
\end{align}
\end{proposition}
\begin{proof}
See proposition \ref{commute}.
\end{proof}
\begin{corollary}
\begin{align}\notag
\lim_{m\to\infty}&\sum_{p=0}^m 2^{n-5/2}\frac{1}{2^{2p-1}}\left(\frac{3}{2}\right)_{2p}\left(\frac{1}{(p!)^2}+2\sum_{k=1}^{\floor{p/2^{n-2}}}(-1)^k\frac{1}{(p-k2^{n-2})!(p+k2^{n-2})!}\right)\\\notag
&=\sum_{j=1}^{2^{n-2}}\left(-j^2+(2^{n-1}+1)j-2^{n-2}\right)\frac{1}{\sin[\frac{(2j-1)\pi}{2^n}]}\\\notag
\end{align}
\end{corollary}
\begin{proof}
Theorem 3.8 in ref. 1 derives the formula for $\zeta(3)$ in the previous proposition from our Proposition 1.1, making the above formula clear.
\end{proof}
\begin{proposition}
\begin{align}\notag
\zeta(5)&=\lim_{n\to\infty}\frac{\pi^5}{93*2^{5n-6}}\\\notag
&\sum_{j=1}^{2^{n-2}}\left(j^4-2(2^{n-1}+1)j^3+(3* 2^{n-1}-1)j^2+2(2^{n-2}+2^{3n-4}+1)j-2^{n-1}(2^{2n-3}+1)\right)\frac{1}{\sin[\frac{(2j-1)\pi}{2^n}]}\\\notag
\end{align}
\end{proposition}
\begin{proof}
See reference 20, combined with proposition \ref{commute}.
\end{proof}
\begin{corollary}
\begin{align}\notag
\lim_{m\to\infty}&\sum_{p=0}^m 2^{n-3/2}3\frac{1}{2^{2p-1}}\left(\frac{5}{2}\right)_{2p}\left(\frac{1}{(p!)^2}+2\sum_{k=1}^{\floor{p/2^{n-2}}}(-1)^k\frac{1}{(p-k2^{n-2})!(p+k2^{n-2})!}\right)\\\notag
&=\sum_{j=1}^{2^{n-2}}\left(j^4-2(2^{n-1}+1)j^3+(3* 2^{n-1}-1)j^2+2(2^{n-2}+2^{3n-4}+1)j-2^{n-1}(2^{2n-3}+1)\right)\frac{1}{\sin[\frac{(2j-1)\pi}{2^n}]}\\\notag
\end{align}
\end{corollary}
\begin{proof}
See reference 20, combined with proposition \ref{commute}.
\end{proof}
\section{Appendix}
Regarding the $2^{n-2}x2^{n-2}$ matrices, $M_n$ which satisfy $$\cos^r(\frac{(2j-1)\pi}{2^n})=\frac{1}{2^{r-1}} \sum_{k=1}^{2^{n-2}}M_n(j,k)\cos(\frac{(2k-1)\pi}{2^n})$$ the Mathematica code for producing the matrix for terms of $M$, for $r>1$, $r$ odd is given below. Note, the $n$ in the code correponds to $2^{n-2}$ in the above discussion.
\begin{lstlisting}
M = IdentityMatrix[n]; For[j = 0, j <= n - 1, j++, 
 M[[1, j + 1]] = 
  Sum[(-1)^m (Binomial[r, (r - 1)/2 - (2 m n + j)] - 
      Binomial[r, (r - 1)/2 - (2 (m + 1) n - j - 1)]), {m, 0, 
    Floor[(r + 1)/(4 n)]}]]; For[i = 2, i <= n, i++, 
 For[j = 1, j <= n, j++, p = 2 i j - i - j + 1;
  k = Floor[(p - 1)/n];
  ind = (-1)^k (p - k*n); 
  M[[i, ind]] = (-1)^Floor[(n + 2 i j - i - j)/(2 n)] M[[1, j]];]];
\end{lstlisting}
The Mathematica code for producing $\cos^r(\frac{(2i-1)\pi}{2^n})$, $r$ odd, $r> 1$, is:
\begin{lstlisting}
crn[i_, n_, r_] := 
 1/2^(r - 1)
   Sum[(-1)^
    Floor[(2^(n - 2) + 2 i j - i - j)/2^(
      n - 1)] Sum[(-1)^
       k (Binomial[r, (r - 1)/2 - ( k 2^(n - 1) + j - 1)] - 
        Binomial[r, (r - 1)/2 - ((k + 1) 2^(n - 1) - j)]), {k, 0, 
      Floor[(r + 1)/2^
        n]}] Cos[( (2 Mod[(-1)^
           Floor[(2 i j - i - j)/2^(
             n - 2)] (Mod[2 i j - i - j, 2^(n - 2)] + 1), 
          2^(n - 2) + 1] - 1) \[Pi])/2^n], {j, 1, 2^(n - 2)}]
\end{lstlisting}
The Mathematica code for producing $S_n$ is:
\begin{lstlisting}
S_n[n_,r_]:=1/2^(r - 1)
  Sum[Sum[(-1)^
    Floor[(2^(n - 2) + 2 i j - i - j)/2^(
      n - 1)] Sum[(-1)^
       k (Binomial[r, (r - 1)/2 - ( k 2^(n - 1) + j - 1)] - 
        Binomial[r, (r - 1)/2 - ((k + 1) 2^(n - 1) - j)]), {k, 0, 
      Floor[(r + 1)/2^
        n]}] Cos[( (2 Mod[(-1)^
           Floor[(2 i j - i - j)/2^(
             n - 2)] (Mod[2 i j - i - j, 2^(n - 2)] + 1), 
          2^(n - 2) + 1] - 1) \[Pi])/2^n], {j, 1, 2^(n - 2)}], {i, 1, 
   2^(n - 2)}]
\end{lstlisting}
Regarding the $2^{n-2}x2^{n-2}$ matrices, $M_n$ which satisfy $$\frac{1}{(\sin(\frac{(2j-1)\pi}{2^n}))^3}=8 \sum_{k=1}^{2^{n-2}}M_n(j,k)\sin(\frac{(2k-1)\pi}{2^n})$$
the Mathematica code for producing the matrix for terms of $M_m$ is given below. 
\begin{lstlisting}
n = 2^(m-2); M = IdentityMatrix[n];
For[i = 1, i <= n, i++, 
 For[j = 1, j <= n, j++, r = Mod[(i + j - 1) (2 i - 1)^(n - 1), 2 n]; 
  M[[i, j]] = ((-1)^((r (2 i - 1) - (i + j - 1))/(2 n))) (1/
      2) (2 n r - r^2 + r - n);
  ]]
\end{lstlisting}
The code for producing $(2^{-3})S_m$ is
\begin{lstlisting}
n=2^(m-2);S=0;
For[i=1,i<=n,i++, 
 For[j=1,j<=n,j++, 
  r=Mod[(i+j-1)(2i-1)^(n-1),2n]; 
  S=S+((-1)^((r(2i-1)-(i+j-1))/(2n)))* 
        (1/4)(2nr-r^2+r-n)Sin[(2j-1)Pi/(4n)];
 ]
]
\end{lstlisting}
Another version:
\begin{lstlisting}
n=2^(m-2);S=0;
For[i=1,i<=n,i++, 
 For[j=1,j<=n,j++, 
  p=i+(j-1)(2(i-1)+1);
  k=Mod[p-1,2n]+1;
  S=S+(-1)^((p-k)/(2n))*
       (1/4)(2nj-j^2+j-n)Sin[(2k-1)Pi/(4n)];
 ]
]
\end{lstlisting}

Regarding the $2^{n-2}x2^{n-2}$ matrices, $M_n$ which satisfy $$\frac{1}{(\sin(\frac{(2j-1)\pi}{2^n}))^5}=32 \sum_{k=1}^{2^{n-2}}M_n(j,k)\sin(\frac{(2k-1)\pi}{2^n})$$
the Mathematica code for producing the matrix for terms of $M_{m+2}$ is given below.
\begin{lstlisting}
 s = 11; For[j = 3, j <= m, j++, s = 2^(3 (j - 1)) + 2 s; ]
n = 2^m; M = IdentityMatrix[n]; M[[1, 1]] = s;
M[[1, n]] = 2^(2*m - 1);
For[i = 2, i < n + 1, i++, 
 M[[1, i]] = 
  M[[1, 1]] + 
   1/24 (-1 + i) (i - 2 n) (i^2 - i (1 + 2 n) + 
      2 (-1 - 6 M[[1, n]] + n^2))]; 
For[i = 2, i <= n, i++, 
  For[j = 1, j <= n, j++, p = i + (j - 1) (2 (i - 1) + 1);
   k = Floor[(p - 1)/n];
   ind = (-1)^k (p - k*n);
   M[[i, ind]] = (-1)^Floor[(p - 1)/(2 n)] M[[1, j]];]] ;
 \end{lstlisting} 
The code for producing $(2^{-5})S_m$ is
\begin{lstlisting}
 s = 11; For[j = 3, j <= m, j++, s = 2^(3 (j - 1)) + 2 s; 
n = 2^m; sum = 0;
For[i = 1, i <= n, i++, 
  For[j = 1, j <= n, j++, p = i + (j - 1) (2 (i - 1) + 1);
   k = Floor[(p - 1)/n];
   ind = (-1)^k (p - k*n); 
   If[ind < 0, ind = ind + 2^m + 1, ind = ind]; 
   sum = sum + (-1)^
       Floor[(p - 1)/2^(m + 1)] (s + 
        1/24 (j - 1) (j - 2^(m + 1)) (j^2 - j (1 + 2^(m + 1)) + 
           2 (-1 - 2^(2 m + 1)))) Sin[(2 ind - 1) \[Pi]/(2^(
          m + 2))];]]; 
 \end{lstlisting} 
Another version:
\begin{lstlisting}
 s = 11; For[j = 3, j <= m - 2, j++, s = 2^(3 (j - 1)) + 2 s; 
sum = 0;
sum= Simplify[
   32 Pi^5/(31*2^(
       m - 1)) Sum[(1/
        24 (j/2^(
         m - 1)) (j/2^(m - 1) - 1) ((j/2^(m - 1))^2 - j/2^(m - 1) + 
          2 (-(1/2)))) Sum[(-1)^
         Floor[(2 i j - i - j)/2^(
           m - 1)] Sin[(2 Mod[(2 i j - i - j), 2^(m - 1)] + 
            1) \[Pi]/(2^m)], {i, 1, 2^(m - 2)}], {j, 1, 2^(m - 2)}]];
 \end{lstlisting}     
 
 Code for producing matricies $M_n$ for even powers of cosine:
 \begin{lstlisting}
 n = 5; r = 16; M = ConstantArray[0, {2^(n - 2), 2^(n - 2)}]; For[
 i = 1, i <= 2^(n - 2), i++, 
 For[k = 0, k <= Floor[(r + 1)/2^n], k++, 
  M[[i, 1]] = 
   M[[i, 1]] + 
    1/2 (-1)^
      k (Binomial[r, r/2 - k 2^(n - 1) ] - 
       Binomial[r, r/2 - ((k + 1) 2^(n - 1) )])]]; For[i = 1, 
 i <= 2^(n - 2), i++, 
 For[m = 2, m <= n - 1, m++, 
  For[j = 1, j <= 2^(n - m - 1), j++, 
   For[k = 0, k <= Floor[(r + 1)/2^n], k++, 
    M[[i, tind[i, j, m, n]]] = 
     M[[i, tind[i, j, m, n]]] + (-1)^
        k (Binomial[r, r/2 - (k 2^(n - 1) + 2^(m - 2) (2 j - 1 ))] - 
         Binomial[r, 
          r/2 - ((k + 1) 2^(n - 1) - 2^(m - 2) (2 j - 1 ))])]; 
   M[[i, tind[i, j, m, n]]] = (-1)^
      Floor[(2^(n - m - 1) + 2 i j - i - j)/2^(n - m)] M[[i, 
       tind[i, j, m, n]]]]]]
 \end{lstlisting}
\begin{section}{{\textbf{Bibliography}}}
\frenchspacing
\begin{itemize}
\item{[1]} Johan W{\"a}stlund,``Summing inverse squares by euclidean geometry'', \\(PDF) http://www.math.chalmers.se/~wastlund/Cosmic.pdf
\item{[2]} 3BlueBrown (YouTube channel video), \"Why is pi here? And why is it squared? A geometric answer to the Basel problem".
\item{[3]} Wolfram Mathworld,``Apery's Constant", https://mathworld.wolfram.com/AperysConstant.html
\item{[4]} C. Nash and D. O'Connor, ``Determinants of Laplacians, the Ray-Singer torsion on lens
spaces and the Riemann zeta function", J. Math. Phys. 36(1995), 1462-1505.
\item{[5]} Maarten J. Kronenburg,``The Binomial Coefficient for Negative Arguments", arXiv:1105.3689 [math.CO],18 May 2011
\item{[6]} Michael Penn, ``Integral Of $(\ln(\cos(x))^3$'', YouTube video.
\item{[7]} Arkadiusz Wesolowski,``OEIS", Oct.16, 2013
\item{[8]} Peter Paule and Markus Schorn, `` A Mathematica Version of Zeilberger's Algorithm for Proving Binomial Coefficient Identities", J. Symbolic Computation (1994) 11, 1-000
\item{[9]} Gosper, Jr., Ralph William "Bill" (January 1978) [1977-09-26]. "Decision procedure for indefinite hypergeometric summation" (PDF). Proceedings of the National Academy of Sciences of the United States of America.
\item{[10]} Petkovšek, Marko; Wilf, Herbert; Zeilberger, Doron (1996). A = B. Home Page for the Book "A=B". A K Peters Ltd. ISBN 1-56881-063-6. Archived 
\item{[11]} Morse, P. M. and Feshbach, H. Methods of Theoretical Physics, Part I. New York: McGraw-Hill, pp. 411-413, 1953.
\item{[12]} "Finite-difference calculus", Encyclopedia of Mathematics, EMS Press, 2001 [1994]
\item{[13]} Charles Jordan “On Stirling’s numbers”, Tohoku Math. Journal, Vol. 37, 1933
\item{[14]} Henry W. Gould,  "Tables of Combinatorial Identities", edited by Jocelyn Quaintance, Vol. 8
\item{[15]} Philippe Flajolet and Robert Sedgewick, "Mellin transforms and asymptotics: Finite differences and Rice's integrals", Theoretical Computer Science 144 (1995) pp 101–124.
\item{[16]} Murray R. Spiegel, ``Schaum's Outline Of Theorey And Problems Of Calculus Of Finite Differences And Difference Equations'',McGraw-Hill, 1971
\item{[17]} R. Kl\'en, M. Visuri and M. Vuorinen, ``On Jordan type inequalities for hyperbolic functions'',J. Inequal. Appl., Vol 2010, Article no. 362548, 2010
\item{[18]} Samuel Greitzer “Many cheerful Facts” Arbelos 4 (1986), no. 5, 14-17
\item{[19]} M. Merca, ''A note on cosine power sums'',J. Integer Seq., 15 (5) (2012), Article 12.5.3
\item{[20]} L. Fairbanks, ``Notes On An Approach to Apery's Constant',  arXiv:2206.11256 [math.NT]
\end{itemize}
\end{section}
\end{flushleft}
\end{document}